\documentclass[leqno,11pt]{amsart}

\usepackage{geometry}

\usepackage[all,cmtip]{xy}
\usepackage{tikz-cd}

\usepackage[all]{xy} 

\usepackage{comment}

\usepackage{mathtools}
\usepackage{amsmath, amssymb, amsfonts, latexsym, mdwlist, amsthm}
\usepackage{subfig}
\usepackage{graphicx}
\usepackage{wrapfig}

\usepackage[bookmarks, colorlinks, breaklinks, pdftitle={},
pdfauthor={}]{hyperref}
\hypersetup{linkcolor=blue,citecolor=blue,filecolor=black,urlcolor=blue}

\usepackage{tikz}
\usetikzlibrary{calc,trees,positioning,arrows,chains,shapes.geometric,%
    decorations.pathreplacing,decorations.pathmorphing,shapes,%
    matrix,shapes.symbols}

\tikzset{
>=stealth',
  punktchain/.style={
    rectangle,
    rounded corners,
    draw=black, thick,
    minimum height=3em,
    text centered,
    on chain},
  line/.style={draw, thick, <-},
  element/.style={
    tape,
    top color=white,
    bottom color=blue!50!black!60!,
    minimum width=8em,
    draw=blue!40!black!90, very thick,
    text width=10em,
    minimum height=3.5em,
    text centered,
    on chain},
  every join/.style={->, thick,shorten >=1pt},
  decoration={brace},
  tuborg/.style={decorate},
  tubnode/.style={midway, right=2pt},
}

\usepackage{paralist}
\setdefaultenum{(a)}{(i)}{}{}
\usepackage{enumitem} 
\usepackage{graphicx}

\usetikzlibrary{patterns}

\renewcommand\_{^{}_}
\renewcommand\;{\hspace{.6pt}}
\newcommand\To{\longrightarrow}
\newcommand\into{\hookrightarrow}
\newcommand{\Into}{\ensuremath{\lhook\joinrel\relbar\joinrel\rightarrow}}
\newcommand\Onto{\longrightarrow\hspace{-5.5mm}\longrightarrow}

\newcommand\PP{\mathbb P}

\newcommand\Q{\mathbb Q}
\newcommand\R{\mathbb R}
\newcommand\N{\mathbb N}
\newcommand\Z{\mathbb Z}
\newcommand\J{\mathsf J\;}

\newcommand\cA{\mathcal A}
\newcommand\cD{\mathcal D}
\newcommand\cH{\mathcal H}
\newcommand\cO{\mathcal O}

\newcommand\ch{\operatorname{ch}}

\newcommand\Pic{\operatorname{Pic}}
\newcommand\Ext{\operatorname{Ext}}
\newcommand\rk{\operatorname{rank}}
\newcommand\Coh{\operatorname{Coh}}
\newcommand\cok{\operatorname{coker}}
\newcommand\js{\operatorname{JS}}
\newcommand\pt{\operatorname{PT}}
\newcommand\dt{\operatorname{DT}}

\newcommand\p{\operatorname{P}}
\newcommand\qe{\operatorname{I}}

\renewcommand\v{\mathsf v}

\newcommand\Ab{\mathcal A_{\;b}}
\newcommand\nubw{\nu\_{b,w}}

\newcommand\al{\alpha}

\renewcommand\({\big(}
\renewcommand\){\big)}
\renewcommand\={\ =\ }

\newcommand\w{\operatorname{w}}

\def\abs#1{\left\lvert#1\right\rvert}

\newcommand\beq[1]{\begin{equation}\label{#1}}
\newcommand\eeq{\end{equation}}
\newcommand\beqa{\begin{eqnarray*}}
\newcommand\eeqa{\end{eqnarray*}}

\makeatletter
\newtheorem*{rep@theorem}{\rep@title}
\newcommand{\newreptheorem}[2]{%
\newenvironment{rep#1}[1]{%
 \def\rep@title{#2 \ref{##1}}%
 \begin{rep@theorem}}%
 {\end{rep@theorem}}}
\makeatother

\newtheorem{Thm}{Theorem}[section]
\newreptheorem{Thm}{Theorem}
\newtheorem{Thm*}{Theorem}
\newtheorem{Prop}[Thm]{Proposition}

\newtheorem{Lem}[Thm]{Lemma}

\newtheorem{Cor}[Thm]{Corollary}
\newreptheorem{Cor}{Corollary}
\newtheorem{Con}[Thm]{Conjecture}

\newtheorem{thm-int}{Theorem}

\theoremstyle{definition}
\newtheorem{Def-s}[Thm]{Definition}
\newtheorem{Def}[Thm]{Definition}
\newtheorem{Rem}[Thm]{Remark}

\newcommand{\ignore}[1]{}

\begin{document}
\title[explicit formulae for DT invariants]{\ \vspace{-15mm} \\ explicit formulae for rank zero DT invariants and \\ the OSV conjecture \vspace{-3mm}}
\author{Soheyla Feyzbakhsh\vspace{-6mm}}
\maketitle
\begin{abstract}
	Fix a Calabi-Yau 3-fold $X$ satisfying the Bogomolov-Gieseker conjecture of Bayer-Macr\`i-Toda, such as the quintic 3-fold. 
	
	By two different wall-crossing arguments we prove two different explicit formulae relating rank 0 Donaldson-Thomas invariants (counting torsion sheaves on $X$ supported on ample divisors) in terms of rank 1 Donaldson-Thomas invariants (counting ideal sheaves of curves) and Pandharipande-Thomas invariants. In particular, we prove a slight modification of Toda's formulation of OSV conjecture for $X$. 
	
    When $X$ is of Picard rank one, we also give an explicit formula for rank two DT invariants in terms of rank zero and rank one DT invariants.

\end{abstract}
\bigskip

\section{Introduction} \label{intro}




Let $(X, \cO_X(1))$ be a smooth polarised Calabi-Yau threefold satisfying the Bogomolov-Gieseker conjecture of \cite{bayer:bridgeland-stability-conditions-on-threefolds,bayer:the-space-of-stability-conditions-on-abelian-threefolds} and let $H \coloneqq c_1(\cO(1))$. Let $\J(\al)\in\Q$ be the Joyce-Song's generalised Donaldson-Thomas invariant \cite{joyce:a-theoery-of-generalised-donadlson-thomas-invariants} counting $H$-Gieseker semistable sheaves of numerical K-theory class $\al$ on $X$. Fix a rank zero, dimension 2 class 
\begin{equation}\label{class-v}
\v = (0, D, \beta, m) \in K(X)
\end{equation}
with $D \neq 0$. In this paper, we relate $\J(\v)$ to Donaldson-Thomas (DT) type invariants on $X$ counting ideal sheaves of curves and Pandharipande-Thomas (PT) stable pairs via two different procedures.

\textbf{Method I:} We do wall-crossing in the space of weak stability condition for the class $\v$ as in \cite{toda:bogomolov-counting}. Then we find an explicit expression of $\J(\v)$ in terms of rank 1 DT invariants and PT invariant. However, the construction works only for specific classes $\v$. 

The PT stable pair invariant $\p_{m_1, \beta_1}$ counts pairs $(F, s)$ consisting of a 1-dimensional pure sheaf $F$ with $(\ch_2(F), \ch_3(F)) = (\beta_1, m_1)$ and a section 
$$
s \colon \cO_X \rightarrow F
$$ 
with zero-dimensional cokernel. Also $\qe_{m_2, \beta_2} 
$ is the rank 1 DT invariant counting ideal sheaves of subschemes $C \subset X$ satisfying 
\begin{equation*}
[C] = \beta_2 \qquad \text{and} \qquad \chi(\cO_C) = m_2.
\end{equation*}
For the rank zero class $\v$ \eqref{class-v}, we define
\begin{equation*}
Q(\v) \coloneqq \frac{1}{2} \left(\frac{D.H^2}{H^3} \right)^2 + 6\left(\frac{\beta.H}{D.H^2}\right)^2 - \frac{12\,m}{D.H^2}\, . 
\end{equation*}
\begin{Thm}\label{thm-rank zero-wall-crossing}
		(i) If $Q(\v) <0$, there is no slope-semistable sheaf of class $\v$. 
		\newline
		(ii) If 
		\begin{equation}\label{bound for Q}
		(H^3)^2\, Q(\v)\ < \  
		D.H^2 + \frac{2}{D.H^2} - \frac{5}{2} - \frac{2}{(D.H^2)^2}\,, 
		\end{equation}
		any slope-semistable sheaf of class $\v$ is slope-stable and 
		\begin{equation*}\label{sum}
		\J(\v) = \big(\#H^2(X,\Z)_{\mathrm{tors}}\big)^2 \sum_{\substack{v_1\, =\, -e^{D_1}(1, 0, -\beta_1, -m_1)\\ v_2 \,=\, e^{D_2}(1, 0, - \beta_2, -m_2)\\
				v_1+v_2\, =\, \v \\
				(D_i,\ \beta_i,\ (-1)^{i+1}m_i) \,\in\, M\left(\v \right)
		}}  (-1)^{\chi(v_2, v_1) -1}\;\chi(v_2, v_1)\, \p_{-m_1, \beta_1}\ \qe_{m_2, \beta_2}\,.
		\end{equation*}
Here $M(\v)$ is a subset of $H^2(X, \Z) \oplus H^4(X, \Z) \oplus H^6(X, \Z)$ described in Definition \ref{Def.M-v} which only depends on $\ch_1(\v).H^2 =D.H^2$.    
\end{Thm}
In \cite[Theorem 3.18]{toda:bogomolov-counting} Toda proved Theorem \ref{thm-rank zero-wall-crossing} for Calabi-Yau 3-folds $X$ with $\Pic(X) = \mathbb{Z}.H$. In this paper, we apply a different wall-crossing argument to manage all Calabi-Yau 3-folds with arbitrary Picard rank.   

Toda's work in \cite{toda:bogomolov-counting} is motivated by Denef--Moore's approach \cite{denef:split-states} toward Ooguri--Strominger--Vafa (OSV) conjecture \cite{vafa:black-hole-attractors} relating black hole entropy and topological string on Calabi-Yau 3-folds. In \cite[Conjecture 1.1]{toda:bogomolov-counting}, he gives a mathematical formulation of OSV conjecture and proves the conjecture when $X$ is of Picard rank one. In Section \ref{section.osv}, as a result of Theorem \ref{thm-rank zero-wall-crossing}, we prove a slight modification\footnote{See Remark \ref{rem-toda} for more details.} of his conjecture for arbitrary Calabi-Yau 3-folds, see Theorem \ref{thm-toda-osv}.

\textbf{Method II.} To find an explicit expression for rank 0 DT invariants in terms of rank 1 DT and PT invariants, we may apply the same technique as \cite{feyz:rank-r-dt-theory-from-1}. Fix $n \gg 0$ and consider a non-zero map $s \colon \cO_X(-n) \xrightarrow{s} F$ for an $H$-Gieseker semistable sheaf $F$ of class $\v$. Then we do wall-crossing on the space of weak stability conditions for class $\v_n \coloneqq \ch(\text{cone}(s))$ \emph{instead of class $\v$}. By \cite[Theorem 2]{feyz:rank-r-dt-theory-from-1}, this gives a universal formula expressing $\J(\v)$ in terms of PT invariants. In this paper we show that the formula can be made explicit when $\Pic(X) = \mathbb{Z}.H$. Note that against Method I, this approach works for any arbitrary class $\v$ of rank zero, but the final relation is not as direct as what we get via Method I. We set 
 \begin{align*}
\pt^{\v, n}(x, y, z) \coloneqq \sum_{\alpha_1\,=\, -e^{k_1H}(1,\, 0,\, -\beta_1,\, m_1) \,\in\, M_{\v, n}} \p_{m_1, \beta_1}\ x^{k_1}\,y^{\beta_1-\frac{1}{2}k_1^2H^2}\,z^{-m_1 -k_1\beta_1.H +\frac{1}{6}k_1^3H^3} \,,
\end{align*}
where $M_{\v, n}$ is a set of rank $-1$ classes in $K(X)$ depending on $\v$ and $n$, see Section \ref{js section}. For any $\mu \in \mathbb{R}$, we consider the generating series
\begin{equation*}
	A(\v_n, \mu) \coloneqq \pt^{\v, n}(x, y, z)  \ \cdot\  
	\prod_{\substack{  
			\al'\, =\, (0,\, k'H,\, \beta',\, m')\, \in\, K(X)\\
			0 \, <\,  k'\, < \, k \\
			\frac{\beta'.H}{k'H^3} \, > \, \mu
	}}
	\exp\left((-1)^{\chi_{\al'}}\,\chi_{\al'}\ \J(\al') \, x^{k'}y^{\beta'}z^{m'}  \right) 
	\end{equation*}
Here $\chi_{\alpha'}$ is a real number described in Definition \ref{Def.A.alpha}. 
\begin{Thm}\label{thm-rank0-n}
	Let $X$ be a smooth Calabi-Yau 3-fold with $\Pic(X) = \mathbb{Z}.H$, and let $\ch_1(\v) = D = kH$. There is $\mu_{\v, n} \in \mathbb{R}$ such that the coefficient of $x^{n+k}y^{\beta-\frac{n^2H^2}{2}}z^{m+ \frac{n^3H^3}{6}}$ in the series
	\begin{equation*}
	\frac{(-1)^{\chi(\cO_X(-n), \v) +1}}{\chi(\cO_X(-n), \v)}
	A(\v_n, \mu_{\v, n})
	\end{equation*}
	is equal to $\J(\v)$. This expresses $\J(v)$ in terms of PT invariants and DT invariants for rank zero classes with lower $\ch_1H^2$, so an inductive argument gives $\J(v)$ in terms of PT invariants.   
\end{Thm}

\subsection*{Rank 2 case} Fix a rank 2 class $\w \in K(X)$ and $n \gg 0$. 
It has been shown in \cite{feyz:rank-r-dt-theory-from-0} that there exists a universal formula expressing $\J(\w)$ in terms of rank 1 DT invariant and rank zero DT invariants. By applying a similar argument as in Method II above, we prove in Section \ref{section.rank 2} that the formula can be made explicit when $\Pic(X) = \Z.H$, see Theorem \ref{thm-rank2}. Combining this with Theorem \ref{thm-rank0-n} gives an expression of rank 2 DT invariants in terms of rank 1 DT invariants and PT invariants.

\subsection*{Outlook}
The main wall-crossing arguments in this paper have been done in Section \ref{section.wall} which are valid for any smooth projective 3-fold satisfying the Bogomolov-Gieseker conjecture of \cite{bayer:bridgeland-stability-conditions-on-threefolds}. Applying Joyce's new wall-crossing formula for Fano 3-folds \cite[Theorem 7.69]{joyce:enumerative-invariants} should hopefully give a Fano version of Theorem \ref{thm-rank zero-wall-crossing} and Theorem \ref{thm-rank0-n} with insertions.  

In this paper, we consider arbitrary Calabi-Yau 3-folds and arbitrary rank 0 or 2 classes in $K(X)$. However, to get a more explicit formula we need to either 
\begin{enumerate*}
	\item restrict ourselves to special types of Calabi-Yau 3-folds like K3$\times$ E \cite{george-k3times-E} or a smooth intersection of quadratic and quartic hypersurfaces in $\PP^5$ \cite{qin:dt-certain}, or 
	\item consider special rank zero or 2 classes $\al \in K(X)$ like $\ch_0(\al)=2$ and $\ch_1(\al)=\ch_2(\al) = 0$ as discussed in \cite{stoppa:d0-d6,toda:rank-2-dt}.   
\end{enumerate*}

Our main technique for controlling walls of instability is the Bogomolov-Gieseker conjecture of \cite{bayer:bridgeland-stability-conditions-on-threefolds}. This is now proved for many 3-folds, including some Calabi-Yau 3-folds \cite{bayer:the-space-of-stability-conditions-on-abelian-threefolds,maciocia:fm-transform-ii}. It is proved for a restricted set of weak stability conditions on quintic threefold \cite{chunyi:stability-condition-quintic-threefold} and a complete intersection of quadratic and quartic hypersurfaces in $\PP^5$ \cite{liu:bg-ineqaulity-quadratic} which are sufficient for our purposes, see Lemma \ref{lem.quintic} and Remark \ref{rem.final}. A weaker version is also proved for double or triple cover CY3 \cite{koseki:double-triple-solids} which is enough for Theorem \ref{thm-rank0-n} and Theorem \ref{thm-rank2} but leads to a weaker version of Theorem \ref{thm-rank zero-wall-crossing}.

\subsection*{Acknowledgements}
I would like to thank Richard Thomas for many useful discussions. I am grateful for comments by Arend Bayer, Davesh Maulik and Yukinobu Toda. The author was supported by Marie Sk\l odowska-Curie individual fellowships 887438 and EPSRC postdoctoral fellowship EP/T018658/1.
\setcounter{tocdepth}{1}
\tableofcontents\vspace{-8mm}

\section{Weak stability conditions \& wall-crossing formula}\label{section weak}
Let $(X, \cO_X(1))$ be a smooth polarised complex projective threefold with bounded derived category of coherent sheaves $\cD(X)$ and Grothendieck group $K(\mathrm{Coh}(X))$. Dividing by the kernel of the Mukai pairing gives the numerical Grothendieck group
\beq{Kdef}
K(X)\ :=\ \frac{K(\mathrm{Coh}(X))}{\ker\chi(\ \ ,\ \ )}\,.
\eeq
Notice $K(X)$ is torsion-free, isomorphic to its image in $H^*(X,\Q)$ under the Chern character. Denoting $H=c_1(\cO_X(1))$, for any $v\in K(X)$ we set
\beqa
\ch\_H(v) &:=& \Big(\!\ch_0(v),\ \tfrac1{H^3}\ch_1(v).H^2,\ \tfrac1{H^3}\ch_2(v).H,\ \tfrac1{H^3}\ch_3(v)\;\!\Big)\ \in\ \Q^4,\\
\ch_H^{\le2}(v) &:=& \Big(\!\ch_0(v),\ \tfrac1{H^3}\ch_1(v).H^2,\ \tfrac1{H^3}\ch_2(v).H\;\!\Big)\ \in\ \Q^3.
\eeqa

We define the $\mu\_H$-slope of a coherent sheaf $E$ to be
$$
\mu\_H(E)\ :=\ \left\{\!\!\!\begin{array}{cc} \frac{\ch_1(E).H^2}{\ch_0(E)H^3} & \text{if }\ch_0(E)\ne0, \\
+\infty & \text{if }\ch_0(E)=0. \end{array}\right.
$$
Associated to this slope every sheaf $E$ has a Harder-Narasimhan filtration. Its graded pieces have slopes whose maximum we denote by $\mu_H^+(E)$ and minimum by $\mu_H^-(E)$.

For any $b \in \mathbb{R}$, let $\cA_{\;b}\subset\cD(X)$ denote the abelian category of complexes
\begin{equation}\label{Abdef}
\Ab\ =\ \big\{E^{-1} \xrightarrow{\ d\ } E^0 \ \colon\ \mu_H^{+}(\ker d) \leq b \,,\  \mu_H^{-}(\cok d) > b \big\}. 
\end{equation}
In particular, setting $\ch^{bH}(E):=\ch(E)e^{-bH}$, each $E\in\cA_{\;b}$ satisfies
\beq{65}
\ch_1^{bH}(E).H^2\=\ch_1(E).H^2-bH^3\ch_0(E) \ge\ 0.
\eeq
By \cite[Lemma 6.1]{bridgeland:K3-surfaces} $\cA_{\;b}$ is the heart of a bounded t-structure on $\cD(X)$. We denote its positive cone by
\beq{CAb}
C(\Ab)\ :=\ \Big\{\sum\nolimits_ia_i[E_i]\ \colon\ a_i\,\in\,\N,\ E_i\,\in\,\Ab\Big\}\ \subset\ K(X).
\eeq
For any $w>\frac12b^2$, we have on $\cA_{\;b}$ the slope function
\begin{equation}\label{noo}
\nubw(E)\ =\ \left\{\!\!\!\begin{array}{cc} \frac{\ch_2(E).H - w\ch_0(E)H^3}{\ch_1^{bH}(E).H^2}
& \text{if }\ch_1^{bH}(E).H^2\ne0, \\
+\infty & \text{if }\ch_1^{bH}(E).H^2=0. \end{array}\right.
\end{equation}
By \cite{bayer:bridgeland-stability-conditions-on-threefolds}\footnote{We use notation from \cite{feyz:curve-counting}; in particular the rescaling \cite[Equation 6]{feyz:curve-counting} of \cite{bayer:bridgeland-stability-conditions-on-threefolds}'s slope function.} $\nubw$ defines a Harder--Narasimhan filtration on $\cA_{\;b}$, and so a \emph{weak stability condition} on $\cD(X)$.

\begin{Def}
	Fix $w>\frac12b^2$. Given an injection $F\into E$ in $\cA_{\;b}$ we call $F$ a \emph{destabilising subobject of} $E$ if and only if
	\beq{seesaw1}
	\nubw(F)\ \ge\ \nubw(E/F),
	\eeq
	and \emph{strictly destabilising} if $>$ holds. 
	We say $E\in\cD(X)$ is $\nubw$-(semi)stable if and only if
	\begin{itemize}
		\item $E[k]\in\cA_{\;b}$ for some $k\in\Z$, and
		\item 
		$E[k]$ contains no (strictly) destabilising subobjects.
	\end{itemize}
\end{Def}


By \cite[Theorem 3.5]{bayer:the-space-of-stability-conditions-on-abelian-threefolds} any $\nubw$-semistable object $E \in \mathcal{D}(X)$ satisfies
\beq{BOG}
\Delta_H(E)\ :=\ \left(\ch_1(E).H^2\right)^2 -2(\ch_2(E).H)\ch_0(E)H^3\ \geq\ 0.
\eeq
Therefore, if we plot the $(b,w)$-plane simultaneously with the image of the projection map
\begin{eqnarray*}
	\Pi\colon\ K(X) \smallsetminus \big\{E \colon \ch_0(E) = 0\big\}\! &\longrightarrow& \R^2, \\
	E &\ensuremath{\shortmid\joinrel\relbar\joinrel\rightarrow}& \!\!\bigg(\frac{\ch_1(E).H^2}{\ch_0(E)H^3}\,,\, \frac{\ch_2(E).H}{\ch_0(E)H^3}\bigg),
\end{eqnarray*}
as in Figure \ref{projetcion}, then $\nubw$-semistable objects $E$ lie outside the
open set
\begin{equation}\label{Udef}
U\ :=\ \Big\{(b,w) \in \mathbb{R}^2 \colon w > \tfrac12b^2  \Big\}
\end{equation}
while $(b,w)$ lies inside $U$. The slope $\nubw(E)$ of $E$ is the gradient of the line connecting $(b,w)$ to $\Pi(E)$ (or $\ch_2(E).H\big/\!\ch_1(E).H^2$ if $\rk(E)=0$).

\begin{figure}[h]
	\begin{centering}
		\definecolor{zzttqq}{rgb}{0.27,0.27,0.27}
		\definecolor{qqqqff}{rgb}{0.33,0.33,0.33}
		\definecolor{uququq}{rgb}{0.25,0.25,0.25}
		\definecolor{xdxdff}{rgb}{0.66,0.66,0.66}
		
		\begin{tikzpicture}[line cap=round,line join=round,>=triangle 45,x=1.0cm,y=0.9cm]
		
		\draw[->,color=black] (-4,0) -- (4,0);
		\draw  (4, 0) node [right ] {$b,\,\frac{\ch_1\!.\;H^2}{\ch_0H^3}$};


		\fill [fill=gray!40!white] (0,0) parabola (3,4) parabola [bend at end] (-3,4) parabola [bend at end] (0,0);
		
		\draw  (0,0) parabola (3.1,4.27); 
		\draw  (0,0) parabola (-3.1,4.27); 
		\draw  (3.8 , 3.6) node [above] {$w= \frac{b^2}{2}$};

		\draw[->,color=black] (0,-.8) -- (0,4.7);
		\draw  (1, 4.1) node [above ] {$w,\,\frac{\ch_2\!.\;H}{\ch_0H^3}$};
		
		\draw [dashed, color=black] (-2.3,1.5) -- (-2.3,0);
		\draw [dashed, color=black] (-2.3, 1.5) -- (0, 1.5);
		\draw [color=black] (-2.6, 1.36) -- (1.3, 3.14);
		
		\draw  (-2.8, 1.8) node {$\Pi(E)$};
		\draw  (-1, 3) node [above] {\Large{$U$}};
		\draw  (0, 1.5) node [right] {$\frac{\ch_2(E).H}{\ch_0(E)H^3}$};
		\draw  (-2.3 , 0) node [below] {$\frac{\ch_1(E).H^2}{\ch_0(E)H^3}$};
		\begin{scriptsize}
		\fill (0, 1.5) circle (2pt);
		\fill (-2.3,0) circle (2pt);
		\fill (-2.3,1.5) circle (2pt);
		\fill (1,3) circle (2pt);
		\draw  (1.2, 2.96) node [below] {$(b,w)$};
		
		\end{scriptsize}
		
		\end{tikzpicture}
		
		\caption{$(b,w)$-plane and the projection $\Pi(E)$ when $\ch_0(E)<0$}
		
		\label{projetcion}
		
	\end{centering}
\end{figure}
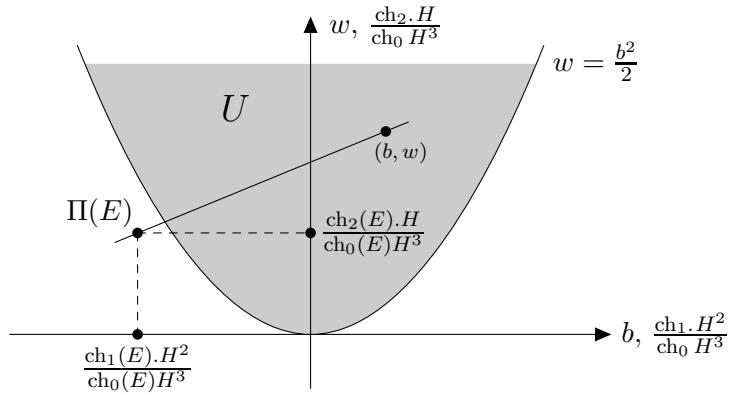

Objects in $\cD(X)$ give the space of weak stability conditions a wall and chamber structure by \cite[Proposition 12.5]{bayer:the-space-of-stability-conditions-on-abelian-threefolds}, as rephrased in \cite[Proposition 4.1]{feyz:thomas-noether-loci} for instance.


\begin{Prop}[\textbf{Wall and chamber structure}]\label{prop. locally finite set of walls}
	Fix $v\in K(X)$ with $\Delta_H(v)\ge0$ and $\ch_H^{\le2}(v)\ne0$. There exists a set of lines $\{\ell_i\}_{i \in I}$ in $\mathbb{R}^2$ such that the segments $\ell_i\cap U$ (called ``\emph{walls of instability}") are locally finite and satisfy 
	\begin{itemize*}
		\item[\emph{(}a\emph{)}] If $\ch_0(v)\ne0$ then all lines $\ell_i$ pass through $\Pi(v)$.
		\item[\emph{(}b\emph{)}] If $\ch_0(v)=0$ then all lines $\ell_i$ are parallel of slope $\frac{\ch_2(v).H}{\ch_1(v).H^2}$.
		\item[\emph{(}c\emph{)}] The $\nubw$-(semi)stability of any $E\in\cD(X)$ of class $v$ is unchanged as $(b,w)$ varies within any connected component (called a ``\emph{chamber}") of $U \smallsetminus \bigcup_{i \in I}\ell_i$.
		\item[\emph{(}d\emph{)}] For any wall $\ell_i\cap U$ there is a map $f\colon F\to E$ in $\cD(X)$ such that
		\begin{itemize}
			\item for any $(b,w) \in \ell_i \cap U$, the objects $E,\,F$ lie in the heart $\cA_{\;b}$,
			\item $E$ is $\nubw$-semistable of class $v$ with $\nubw(E)=\nubw(F)=\,\mathrm{slope}\,(\ell_i)$ constant on the wall $\ell_i \cap U$, and
			\item $f$ is an injection $F\into E $ in $\cA_{\;b}$ which strictly destabilises $E$ for $(b,w)$ in one of the two chambers adjacent to the wall $\ell_i$.
			\hfill$\square$
			
		\end{itemize} 
	\end{itemize*} 
\end{Prop}

In this paper, we always assume $X$ satisfies the conjectural Bogomolov-Gieseker inequality of Bayer-Macr\`i-Toda \cite{bayer:bridgeland-stability-conditions-on-threefolds}.  In the form of \cite[Conjecture 4.1]{bayer:the-space-of-stability-conditions-on-abelian-threefolds}, rephrased in terms of the rescaling \cite[Equation 6]{feyz:curve-counting}, it is the following.

\begin{Con}[\textbf{Bogomolov-Gieseker inequality}]\label{conjecture}
	For any $(b,w)\in U$ and $\nubw$-semistable $E\in\cD(X)$, we have the inequality
	\begin{equation}\label{quadratic form}
	B_{b, w}(E)\ :=\ (2w-b^2)\Delta_H(E) + 4\big(\!\ch_2^{bH}(E).H\big)^2 -6\(\!\ch_1^{bH}(E).H^2\)\ch_3^{bH}(E)\ \geq\ 0.
	\end{equation}
\end{Con}

Multiplying out and cancelling we find that $B_{b, w}$ is actually linear in $(b,w)$:
\beq{boglinear}
\tfrac12B_{b, w}(E)\=\(C_1^2-2C_0C_2\)w+\(3C_0C_3-C_1C_2\)b+(2C_2^2-3C_1C_3),
\eeq
where $C_i:=\ch_i(E).H^{3-i}$.
The coefficient of $w$ is $\ge0$ by \eqref{BOG}. When it is $>0$ the Bogomolov-Gieseker inequality \eqref{quadratic form} says that $E$ can be $\nubw$-semistable only above the line $\ell_f(E)$ defined by the equation $B_{b, w}(E)=0$. When $\ch_0(E)\neq0\neq\ch_1(E).H^2$ we can rearrange to see $\ell_f(E)$ is the line through the points $\Pi(E)$ and 
\begin{equation}\label{pi'}
\Pi'(E)\ :=\, \left(\frac{2\ch_2(E).H}{\ch_1(E).H^2}\,, \ \frac{3\ch_3(E)}{\ch_1(E).H^2}\right).    
\end{equation}

\subsection*{Tilt (Gieseker) stability}
Given a sheaf of class $\alpha\in K(X)$ with Hilbert polynomial 
$$
P_\alpha(t)\ :=\ \chi\(\alpha(t)\)\=a_dt^d+a_{d-1}t^{d-1}+\dots+a_0,
$$
where $d\le3$ and $a_d\ne0$. Its \emph{reduced} Hilbert polynomial is
$$
p_\alpha(t)\ :=\ \frac{P_\alpha(t)}{a_d}\,.
$$
Following Joyce \cite[Section 4.4]{joyce:configurations-iii} we introduce a total order $\prec$ on $\{$monic polynomials$\}\sqcup\{0\}$ by saying $p\prec q$ if and only if
\begin{itemize}
	\item[(i)] deg$\,p>\ $deg$\,q$, or
	\item[(ii)] deg$\,p=\,\;$deg$\,q$ and $p(t)<q(t)$ for $t\gg0$.
\end{itemize}
Our convention is that $\deg0=0$. Then $E\in\mathrm{Coh}\;(X)$ is called Gieseker (semi)stable if for all exact sequences $0\to A\to E\to B\to0$ in Coh$\;(X)$ we have
$$
p\_{\;[A]}\ \,(\preceq)\,\ p\_{\;[B]}.
$$
Here $(\preceq)$ means $\prec$ for stability and $\preceq$ (which is $\prec$ or $=$) for semistability. In particular (i) ensures that Gieseker semistable sheaves are pure. 

Discarding the constant term of the Hilbert polynomial before dividing by its top coefficient as before gives
\beq{tidef}
\widetilde p_\alpha(t)\ :=\ p_\alpha(t)-\frac{a_0}{a_d}\=
t^d+\dots+\frac{a_1}{a_d}t.
\eeq
It depends only on $\ch_H^{\leq 2}(\alpha)$. Then \emph{tilt (semi)stability} on Coh$(X)$ is defined by the inequalities
$$
\widetilde p\_{\;[A]}\ \,(\preceq)\,\ \widetilde p\_{\;[B]}
$$
for all exact sequences of sheaves $0\to A\to E\to B\to0$. 

\subsection*{Large volume limit} We can improve on the local finiteness of walls by showing we have finiteness as $w\to\infty$. This gives, for each fixed $v\in K(X)$, a \emph{large volume chamber} $\subseteq U$ in which there are no walls for $v$, so the $\nubw$-(semi)stability of objects of class $v$ is independent of $w\gg0$, see for instance \cite[Proposition 1.3]{feyz:rank-r-dt-theory-from-1}. Moreover, large volume stability coincides with classical stability.  

\begin{Lem}\label{lem.large volume limit}
	Take an object $E \in \cA_b$ for some $b \in \mathbb{R}$.  
	\begin{enumerate*}
		\item Suppose $\ch_0(E) \geq 0$ and $b < \mu_H(E)$.
		Then $E$ is $\nu\_{b,w}$-(semi)stable for $w \gg 0$ if and only if $E$ is a tilt-(semi)stable sheaf. 
		\item If $\ch_0(E) = -1$ and $\mu_H(E) < b$, then $E$ is $\nu_{b,w}$-semistable for $w \gg 0$ if and only $E^{\vee} \otimes (\det(E))^{-1}[1]$ is a stable pair, i.e. isomorphic to a 2-term complex $\cO_X \xrightarrow{s} F$ in $\cD(X)$ (with $\cO_X$ in degree 0) such that  
		\begin{itemize}
			\item $F$ is a \emph{pure} 1-dimensional sheaf, and 
			\item $s \colon \cO_X \rightarrow F$ has zero-dimensional cokernel. 
		\end{itemize}
		In particular, in this case $E$ is $\nu_{b,w}$-semistable for $w \gg 0$ if and only if it is $\nu_{b, w}$-stable.   
	\end{enumerate*}
\end{Lem}  
\begin{proof}
	Part (a) follows by the same argument as in \cite[Proposition 14.2]{bridgeland:K3-surfaces}. Part (b) is proved in \cite[Section 3]{toda:bogomolov-counting} and \cite[Lemma A.2 and Lemma A.3]{feyz:rank-r-dt-theory-from-1}.  
\end{proof}

\subsection*{Wall-crossing formula}\label{wcross} Now assume $X$ is a Calabi-Yau 3-fold, i.e. $K_X\cong\cO_X$ and $h^1(X, \cO_X) =0$. For any two objects $E, F \in \cD(X)$, the Hirzebruch--Riemann--Roch Theorem implies that the Euler form $\chi([E], [F]) = \sum_{i} (-1)^i\dim_{\mathbb{C}}\Ext^i(E, F)$ is given in terms of their Chern characters by
\begin{align*}
\chi([E_1], [E_2]) =\ & \ch_0(E_1)\ch_3(E_2) - \ch_0(E_2)\ch_3(E_1) + \ch_1(E_2)\ch_2(E_1) -\ch_1(E_1)\ch_2(E_2)\\
& + \frac{1}{12}c_2(X)\big(\ch_0(E_1)\ch_1(E_2) -\ch_0(E_2)\ch_1(E_1) \big). 
\end{align*}
For any $(b,w) \in U$ and any class $v\in K(X)$ with $\nubw(v)<+\infty$, we can apply the work of Joyce-Song \cite{joyce:a-theoery-of-generalised-donadlson-thomas-invariants} to define generalised DT invariants 
\begin{equation*}
\J_{b,w}(v)\ \in\ \Q
\end{equation*}
counting $\nubw$-semistable objects $E \in \cA_b$ of class $v$. In \cite[Section 4 and Appendix C]{feyz:rank-r-dt-theory-from-0}, these invariants have been described and shown that they are independent of the choice of $(b,w)$ inside each chamber described in Proposition \ref{prop. locally finite set of walls} for class $v$. Also the Joyce-Song wall crossing formula applies to the $\J_{b,w}(v)$ under the same $\nubw(v)<+\infty$ condition.


Suppose $\ell$ is a non-vertical wall for a class $v \in K(X)$ with $\ch_1(v)H^2 >0$ if $\ch_0(v) = 0$. Thus $\ell$ passes through $\Pi(v)$ if $\ch_0(v) \neq 0$; otherwise is of gradient $\ch_2(v).H\big/\!\ch_1(v).H^2$. Let $(b,w_0)$ be a point on the line segment $\ell \cap U$. 
By the local finiteness of walls of Proposition \ref{prop. locally finite set of walls} we may choose
\beq{just}
(b,w^\pm)\,\in\,U\,\text{ just above and below the wall }\ell,
\eeq
in the sense that $(b,w^\pm)\not\in\ell$ and between $(b,w_-)$ and $(b,w_+)$ there are no walls for $v$, \emph{nor any of its finitely many semistable factors}, except for $\ell$. Then the wall crossing formula \cite[Equation (5.13)]{joyce:a-theoery-of-generalised-donadlson-thomas-invariants} gives
\begin{align}\label{wcf}
\begin{gathered}
\J_{b, w^-}(v)= \!\!\!\!\!\!
\sum_{\begin{subarray}{l}q\ge 1,\ \alpha_1,\ldots,\alpha_q\in
	C(\cA_b):\\ \alpha_1+\cdots+\alpha_q=v\end{subarray}}\,\,\,\,
\sum_{\begin{subarray}{l}\text{connected, simply-connected digraphs $\Gamma \colon$}\\
	\text{vertices $\{1,\ldots,q\}$, edge $i \rightarrow j$ implies $i<j$}\end{subarray}}\\
\frac{(-1)^{q-1 + \sum_{1 \leq i < j \leq q}\chi(\alpha_i, \alpha_j)}}{2^{q-1}}\, U\left(\alpha_1,\ldots,\alpha_q;(b,w^+),(b,w^-)\right) \!\!\!\!\!
\prod_{\text{edges $i \rightarrow j$ in $\Gamma$}}\!\!\!\!\!
\chi(\alpha_i,\alpha_j) \prod_{i=1}^q\J_{b, w^+}(\alpha_i).
\end{gathered}
\end{align}
Here $C(\Ab)$ is the positive cone \eqref{CAb} and the coefficients $U\left(\alpha_1,\ldots,\alpha_q;(b,w^+),(b,w^-)\right)$ are defined as follows.

\begin{Def}{\cite[Definition 3.12]{joyce:a-theoery-of-generalised-donadlson-thomas-invariants}}\label{Def-U}
	 Take two points $(b, w_1), (b, w_2) \in U$ 
	and let $q\ge 1$ and $\alpha_1,\ldots,\alpha_q\in C(\cA_b)$. If for all
	$i=1,\ldots,q-1$ we have either
	\begin{itemize}
		\setlength{\itemsep}{0pt}
		\setlength{\parsep}{0pt}
		\item[(a)] $\nu\_{b, w_1}(\alpha_i)\le\nu\_{b, w_1}(\alpha_{i+1})$ and
		$\nu\_{b, w_2}(\alpha_1+\cdots+\alpha_i)>\nu\_{b, w_2}(\alpha_{i+1}+\cdots+\alpha_q)$ or
		\item[(b)] $\nu\_{b, w_1}(\alpha_i)>\nu\_{b, w_1}(\alpha_{i+1})$ and~
		$\nu\_{b, w_2}(\alpha_1+\cdots+\alpha_i)\le\nu\_{b, w_2}(\alpha_{i+1}+\cdots+\alpha_q)$,
	\end{itemize}
	then define $S\big(\alpha_1,\ldots,\alpha_q;(b, w_1),(b, w_2)\big)=(-1)^r$ where $r$ is the number
	of $i=1,\ldots,q-1$ satisfying (a). Otherwise define
	$S\big(\alpha_1,\ldots,\alpha_q;(b, w_1),(b, w_2)\big)=0$. Now
	define
	\begin{align}
	&U\big(\alpha_1,\ldots,\alpha_q;(b, w_1),(b, w_2)\big)=
	\label{u}\\
	&\sum_{\begin{subarray}{l} \phantom{wiggle}\\
		1\le p\le t\le q, \;\  0=a_0<a_1<\cdots<a_t=q,\;\ 
		0=b_0<b_1<\cdots<b_{p}=t:\\
		\text{Define $\mathcal{B}_1,\ldots,\mathcal{B}_t\in C(\cA)$ by
			$\mathcal{B}_i=\alpha_{a_{i-1}+1}+\cdots+\alpha_{a_i}$.}\\
		\text{Define $\gamma_1,\ldots,\gamma_{p}\in C(\cA_b)$ by
			$\gamma_i=\mathcal{B}_{b_{i-1}+1}+\cdots+\mathcal{B}_{b_i}$.}\\
		\text{Then $\nu\_{b, w_1}(\beta_i)=\nu\_{b,w_1}(\alpha_j)$, $i=1,\ldots,t$,
			$a_{i-1}<j\le a_i$,}\\
		\text{and $\nu\_{b,w_2}(\gamma_i)=\nu\_{b,w_2}(\alpha_1+\cdots+\alpha_n)$,
			$i=1,\ldots,p$}
		\end{subarray}
		\!\!\!\!\!\!\!\!\!\!\!\!\!\!\!\!\!\!\!\!\!\!\!\!\!\!\!\!\!\!\!\!\!
		\!\!\!\!\!\!\!\!\!\!\!\!\!\!\!\!\!\!\!\!\!\!\!\!\!\!\!\!\!\!\!\!\!
		\!\!\!\!\!\!\!\!\!\!\!\!\!\!\!\!\!\!\!\!}
	\begin{aligned}[t]
	\frac{(-1)^{p-1}}{p}\cdot\prod\nolimits_{i=1}^{p}S\big(\mathcal{B}_{b_{i-1}+1},
	\mathcal{B}_{b_{i-1}+2},\ldots,\mathcal{B}_{b_i}; (b, w_1),(b, w_2)\big)&\\
	\cdot\prod_{i=1}^t\frac{1}{(a_i-a_{i-1})!}&\,.
	\end{aligned}
	\nonumber
	\end{align}
\end{Def}

One can easily check that the term $U\left(\alpha_1,\ldots,\alpha_q; (b, w^+), (b,w^-)\right)$ is zero unless there is a $\nu\_{b,w_0}$-semistable object $E$ of class $v$ with $\nu\_{b,w_0}$-semistable factors of classes $\alpha_1,\dots,\alpha_q$. 
The formula reflects the different Harder-Narasimhan filtrations of $E$ on the two sides of the wall, and then further filtrations of the semistable Harder-Narasimhan factors by semi-destabilising subobjects. 

The complicated formula \eqref{wcf} can be simplified when $q = 2$. Suppose $\nu\_{b,w_0}(\alpha_1)=\nu\_{b,w_0}(\alpha_2)$,
\begin{equation*}
\nu\_{b,w^+}(\alpha_1)>\nu\_{b,w^+}(\alpha_2) \qquad \text{and} \qquad \nu\_{b,w^-}(\alpha_1)<\nu\_{b,w^-}(\alpha_2). 
\end{equation*}
In the definition of $U(\alpha_1, \alpha_2;\, (b,w^+) , (b, w^-))$, we must have $p =1$ and $t=2$, thus 
\begin{equation*}
U(\alpha_1, \alpha_2;\, (b,w^+) , (b, w^-)) = S(\alpha_1, \alpha_2;\, (b,w^+) , (b, w^-)) = 1
\end{equation*}
and 
\begin{equation*}
U(\alpha_2, \alpha_1;\, (b,w^+) , (b, w^-)) = S(\alpha_2, \alpha_1;\, (b,w^+) , (b, w^-)) = -1
\end{equation*}
Thus the summation of coefficient for these two factors in \eqref{wcf} is 
\begin{equation}\label{sentence-2-terms}
(-1)^{\chi(\alpha_1,\alpha_2)+1}\chi(\alpha_1,\alpha_2).
\end{equation}
Let 
\beq{Jtilt}
\J_{\mathrm{ti}}(\alpha)\ \in\ \Q
\quad\text{and}\quad \J(\alpha)\ \in\ \Q 
\eeq
count tilt (Gieseker)-semistable sheaves of class $\alpha$, see \cite[Section 4]{feyz:rank-r-dt-theory-from-0} for more details. If all Gieseker semistable sheaves of class $\alpha\in K(X)$ are Gieseker stable then $\J(\alpha)\in\Z$ is the original DT invariant defined in \cite{thomas:holomorphic-invariants}. We know tilt stability \emph{dominates} Gieseker stability in the sense of \cite[Definition 4.10]{joyce:configurations-iii}. Thus we may apply Joyce-Song wall-crossing formula:
\begin{align}\label{wcf-tilt-gieseker}
\begin{gathered}
\J_{\mathrm{ti}}(v)= \!\!\!\!\!\!
\sum_{\begin{subarray}{l}q\ge 1,\ \alpha_1,\ldots,\alpha_q\,\in\,
	C(\Coh(X)):\\ \alpha_1+\cdots+\alpha_q=v\end{subarray}}\,\,\,\,
\sum_{\begin{subarray}{l}\text{connected, simply-connected digraphs $\Gamma \colon$}\\
	\text{vertices $\{1,\ldots,q\}$, edge $i \rightarrow j$ implies $i<j$}\end{subarray}}\\
\frac{(-1)^{q-1 + \sum_{1 \leq i < j \leq k}\chi(\alpha_i, \alpha_j)}}{2^{q-1}}\, U\left(\alpha_1,\ldots,\alpha_q;\mathrm{Gi}, \mathrm{ti}\right) \!\!\!\!\!
\prod_{\text{edges $i \rightarrow j$ in $\Gamma$}}\!\!\!\!\!
\chi(\alpha_i,\alpha_j) \prod_{i=1}^q\J(\alpha_i).
\end{gathered}
\end{align}
Here $U\left(\alpha_1,\ldots,\alpha_k;\mathrm{Gi}, \mathrm{ti}\right)$ is defined as in Definition \ref{Def-U} by replacing $\nu\_{b, w_1}(\alpha)$ by $p\_{[\alpha]}$  and $\nu\_{b,w_2}(\alpha)$ by $\widetilde{p}\_{[\alpha]}$ for any $\alpha \in K(X)$.   


\section{Walls for rank-zero classes}\label{section.wall}
Let $(X, \cO_X(1))$ be a smooth polarised complex projective threefold with $H = c_1(\cO_X(1))$. Fix a rank zero class $\v = (0, D, \beta, m) \in K(X)$ with 
\begin{equation*}
\ch_H(\v) = (0, k, s, d)
\end{equation*} 
satisfying inequality \eqref{bound for Q}, i.e.
\begin{equation}\label{equiavalent bound for Q}
Q(v) < \frac{1}{2}k^2 - \frac{1}{2} \left(k-\frac{1}{H^3} + \frac{2}{k(H^3)^2} \right)^2
\end{equation}
where 
\begin{equation*}
Q(\v) = \frac{1}{2}k^2 +6\frac{s^2}{k^2} -12 \frac{d}{k}\,.
\end{equation*} 
 In this section, we analyse walls in the space of weak stability conditions $U$ for class $\v$. By Proposition \ref{prop. locally finite set of walls}, walls for class $\v$ are parallel lines of slope $\frac{s}{k} = \frac{\beta.H}{D.H^2}$.  

\begin{Prop}\label{prop.destabilising}
	Any wall $\ell$ for class $\v$ lies above or on the line $\ell_{f}$ with equation 
	\begin{equation}\label{line ell-f}
	w = \frac{s}{k} b + \frac{k^2}{8} - \frac{s^2}{2k^2} - \frac{1}{4}Q(\v).
	\end{equation}  
	Let $E' \rightarrow E \rightarrow E''$ be a destabilising sequence for an object $E$ of class $\v$ along a wall $\ell$, then 
	\begin{itemize*}
		\item one of the destabilising factors $E_1$ (which is either $E'$ or $E''$) is of Chern character $v_1 = -e^{D_1}(1, 0, -\beta_1, -m_1)$ such that $\big(E_1 \otimes \det E_1\big)^{\vee}[1]$ is a stable pair,
		\item the other factor $E_2$ is a torsion-free sheaf of Chern character $v_2 = e^{D_2}(1, 0, -\beta_2, -m_2)$. 
	\end{itemize*}
	Moreover, both $E_1$ and $E_2$ are $\nu_{b,w}$-stable for any $(b,w) \in \ell \cap U$. 
\end{Prop}
\begin{proof}
	By \eqref{boglinear}, if there is a $\nu_{b,w}$-semistable object of class $v$ for some $(b,w) \in U$, then
	\begin{align}\label{bg-conjecture-our use}
	0 \leq \ \frac{1}{2(kH^3)^2}B_{b,w}(\v) =\ & w -b \frac{s}{k} + 2 \frac{s^2}{k^2} -3 \frac{d}{k} \\
	=\ &  w -b \frac{s}{k} -\frac{k^2}{8} +  \frac{s^2}{2k^2} + \frac{1}{4}Q(\v).\nonumber 
	\end{align} 
	Thus $\ell_f$ is the line given by $B_{b,w}(\v) = 0$ and so the first claim follows. The line $\ell_f$ intersects $\partial U$ at two points with $b$-values $b_2 <b_1$, 
	\begin{equation*}
	b_1, b_2 = \frac{s}{k} \pm \sqrt{\frac{k^2}{4} - \frac{1}{2}Q(v) }.
	\end{equation*}
	Then \eqref{bound for Q} or equivalently \eqref{equiavalent bound for Q} gives 
	\begin{align}\label{lower}
	b_1 -b_2 = \sqrt{k^2 -2Q(v) }\ >\ k- \frac{1}{H^3} + \frac{2}{k(H^3)^2}\, .  
	\end{align}
	Since the wall $\ell$ lies above $\ell_f$, the destabilising factors $E_1$ and $E_2$ satisfy
	\begin{equation}\label{b}
	\mu_H^+(\cH^{-1}(E_i)) \leq b_2 \qquad \text{and} \qquad b_1 \leq \mu_H^-(\cH^0(E_i)).
	\end{equation}
	Therefore 
	\begin{align}\label{b-difference}
	kH^3 =\ & \ch_1(\cH^{0}(E_1))H^2 + \ch_1(\cH^{0}(E_1))H^2 - \ch_1(\cH^{-1}(E_1))H^2- \ch_1(\cH^{-1}(E_1))H^2\nonumber\\
	\geq \ & b_1H^3\big(\ch_0(\cH^{0}(E_1)) + \ch_0(\cH^{0}(E_2)) \big) -b_2H^3 \big(\ch_0(\cH^{-1}(E_1)) + \ch_0(\cH^{-1}(E_2)) \big)\nonumber\\
	= \ & (b_1-b_2)H^3\big(\ch_0(\cH^{0}(E_1)) + \ch_0(\cH^{0}(E_2)) \big). 
	\end{align}
	The last inequality comes from $\rk(E) = \rk(E_1) +\rk(E_2) = 0$. Combining this with \eqref{lower} implies that
	\begin{equation*}
	\ch_0(\cH^{-1}(E_1)) + \ch_0(\cH^{-1}(E_2)) = \ch_0(\cH^{0}(E_1)) + \ch_0(\cH^{0}(E_2)) \leq 1
	\end{equation*}
	Therefore, one of the factors $E_1$ is of rank $-1$ with $\cH^{-1}(E_1)$ of rank one and $\cH^0(E_1)$ of rank zero; and the other factor $E_2$ is a sheaf of rank one. 
	
	We claim 
	\begin{equation}\label{lb}
	\mu_H(E_2) -b_1 < \frac{1}{H^3}.
	\end{equation}
	Otherwise, \eqref{b} gives   
	\begin{equation*}
	b_1-b_2 \leq \mu_H(E_2) - \frac{1}{H^3} - \mu_H(\cH^{-1}(E_1)).
	\end{equation*}   	
	Since $\cH^{0}(E_1)$ is of rank zero, we have $\mu_H(E_1) \leq \mu_H(\cH^{-1}(E_1))$, thus
	\begin{equation*}
	b_1-b_2 \leq \mu_H(E_2) -\frac{1}{H^3} -\mu_H(E_1) = k- \frac{1}{H^3}. 
	\end{equation*}
	The last equality comes from $\rk(E_2) = -\rk(E_1) = 1$. But the above is not possible by \eqref{lower}. 
	
	Since $k > 0$ and $k \in \frac{1}{H^3}\Z$, \eqref{lower} gives $b_1-b_2 > \frac{1}{H^3}$. Therefore \eqref{lb} implies that the line $\ell_f$ intersects the vertical line $b = \mu_H(E_2) - \frac{1}{H^3}$ at a point inside $U$. Since the wall $\ell$ lies above $\ell_f$, the same holds for $\ell$. Thus \cite[Lemma 8.1]{feyz:thomas-noether-loci} implies that $E_2$ is $\nu_{b_0, w}$-stable for $b_0 = \mu_H(E_2) - \frac{1}{H^3}$ and all $w > \frac{b_0^2}{2}$, hence by Lemma \ref{lem.large volume limit}, $E_2$ is tilt-stable and so torsion-free.

	By applying a similar argument as above, one can show 
	\begin{equation}
	b_2 - \mu_H(E_1) < \frac{1}{H^3}.
	\end{equation}
	Thus \cite[Lemma 8.2]{feyz:thomas-noether-loci} implies that $E_1$ is stable along the wall $\ell$ and in the large volume limit, so the claim follows from Lemma \ref{lem.large volume limit}.
\end{proof}

For any sheaf $E$ of rank zero, we define $\nu\_H$-slope as 
\begin{equation}\label{nuslope}
\nu\_H(E)\ :=\ \left\{\!\!\begin{array}{cc} \frac{\ch_2(E).H}{\ch_1(E).H^2} & \text{if }\ch_1(E).H^2\ne0, \\
+\infty & \text{if }\ch_1(E).H^2=0. \end{array}\right.
\end{equation}
We say that a sheaf $E$ of rank zero is slope (semi)stable if for all non-trivial quotient sheaves $E\to\hspace{-3mm}\to E'$ one has $\nu\_H(E)\,(\leq)\, \nu\_H(E')$. 

Hence a sheaf $E$ of class $\v$ is slope(semi)-stable if and only if it is tilt-(semi)stable. Thus Lemma \ref{lem.large volume limit} implies that $E$ is $\nu_{b,w}$-(semi)stable for $w \gg 0$ if and only if it is slope-(semi)stable. The following proves the first part of Theorem \ref{thm-rank zero-wall-crossing}.
\begin{Lem}\label{lem-first part of theorem 1}
	If $Q(\v) <0$, there is no slope-semistable sheaf of class $\v$. 
\end{Lem}
\begin{proof}
	Let $\ell_{\v}$ be the line of slope $\frac{s}{c}$ which intersects $\partial U$ at two points with $b$-values $b_{\v} < a_{\v}$ such that $a_{\v} -b_{\v} =k$. It is of equation
	\begin{equation}\label{equation-}
	w = \frac{s}{k} b + \frac{k^2}{8} - \frac{s^2}{2k^2}\ . 
	\end{equation}
	The same argument as in the proof of \cite[Lemma B.3]{feyz:rank-r-dt-theory-from-0} implies that there is no wall for class $\v$ above $\ell_{\v}$. Thus if there is a slope-semistable sheaf of class $\v$, then it is $\nu_{b,w}$-semistable for any $(b,w) \in U$ above $\ell_{\v}$. Hence the final line $\ell_f$ of equation \eqref{line ell-f} must lie on or below $\ell_{\v}$. Comparing its equation \eqref{line ell-f} with \eqref{equation-} gives $Q(\v) \geq 0$.       
\end{proof}
Applying a similar argument as in \cite[Section 2]{feyz:curve-counting} implies the next lemma. 
\begin{Lem}\label{lem-no-strictly-semistable}
	Any slope-semistable sheaf of class $\v$ is slope-stable. 
\end{Lem}
\begin{proof}
	Suppose there is a strictly slope semistable sheaf $E$ of class $\v$. We know $E$ is $\nu_{b,w}$-semistable for $w \gg 0$. When we move down it hits a wall $\ell$ which lies above $\ell_f$. by Proposition \ref{prop.destabilising}, the destabilising sequence is of the form $E_2 \hookrightarrow E \twoheadrightarrow E_1$ such that both $E_1$ is a rank one torsion-free sheaf. Note that $E_1$ is a quotient object as it destabilises $E$ below the wall $\ell$. We also know $E_1$ and $E_2$ are $\nu_{b_0,w_0}$-stable for any $(b_0,w_0) \in \ell \cap U$.

	Suppose $E \twoheadrightarrow E'$ is a proper quotient sheaf with $\nu\_H(E') = \nu\_H(E)$. Since rank$\,E'=0=\,$rank$\,E$ the formula \eqref{noo} gives
	$$
	\nu\_{b,w}(E')\ =\ \frac{\ch_2(E')H}{\ch_1(E')H^2}\ =\ \frac{s}{k}\ =\ \nu\_{b,w}(E)
	$$
	for all $(b,w) \in U$. Since all torsion sheaves are in $\cA_{\;b_0}$, $E'$ is a quotient of $E$ in the abelian category $\cA_{\;b_0}$, and any quotient of $E'$ in $\mathcal{A}_{b_0}$ is also a quotient of $E$. Therefore $E'$ is also $\nu_{b_0, w_0}$-semistable for $(b_0,w_0) \in \ell \cap U$. 
	
	Since $E_2$ is $\nu_{b_0,w_0}$-stable, the composition
	$$
	E_2 \Into E \Onto E'
	$$
	in $\cA_{b_0}$ must either be zero or injective. And it cannot be zero, because this would give a surjection $E_1 \twoheadrightarrow E'$ in $\mathcal{A}_{b_0}$, contradicting the $\nu_{b_0,w_0}$-stability of $E_1$. So it is injective. Let $C$ denote its cokernel in $\cA_{b_0}$, sitting in a commutative diagram
	\begin{equation*}
	\xymatrix@R=16pt{
		\,E_2\,\ar@{^{(}->}[r]\ar@{=}[d] & E \ar@{->>}[r]\ar@{->>}[d]&E_1\ar@{->>}[d]\\
		\,E_2\,\ar@{^{(}->}[r] &E'\ar@{->>}[r]&\,C.}
	\end{equation*}
	Since $E'$ and $E_2$ are $\nu_{b_0,w_0}$-semistable of the same phase, $C$ is also $\nu_{b_0,w_0}$-semistable.
	Therefore the right hand surjection contradicts the $\nu_{b_0,w_0}$-stability of $E_1$. 
\end{proof}

The next step is to analyse Chern character of the destabilising factors along a wall for class $\v$. 

\begin{Def}\label{Def.M-v}
Let $M(\v)$ be the set of all classes $(D',\, \beta',\, m') \in H^2(X, \Z) \oplus H^4(X, \Z) \oplus H^6(X, \Z)$ such that 
\begin{equation*}\label{M}
\frac{1}{2}\left(\frac{D'H^2}{H^3}\right)^2 -  \frac{D'^2H}{2H^3} + \frac{\beta'.H}{H^3} \,\leq\, \frac{DH^2}{2(H^3)^2} - \frac{1}{(H^3)^2} \qquad \text{and} \qquad m' \leq 
\frac{DH^2(DH^2 +H^3)}{6(H^3)^2}. 
\end{equation*}
\end{Def}

\begin{Prop}\label{prop.calsses of factors}
	The destabilising classes $v_i = (-1)^ie^{D_i}(1, 0, -\beta_i, -m_i)$ for $i=1, 2$ in Proposition \ref{prop.destabilising} satisfy 
	\begin{equation}\label{beta-i}
	\frac{1}{2}\left(\frac{D_iH^2}{H^3}\right)^2 -  \frac{D_i^2H}{2H^3} + \frac{\beta_i.H}{H^3} \leq \frac{1}{2}\left(k^2 -k \sqrt{k^2-2Q(\v)}\right), 
	\end{equation}
	and
	\begin{equation}\label{m-i}
	(-1)^{i+1}m_i \leq\  \frac{2}{3}\beta_i.H \left( \beta_i.H + \frac{1}{2H^3}\right). 
	\end{equation}
	In particular $(D_i, \beta_i, (-1)^{i+1}m_i) \in M(\v)$.  
\end{Prop}
\begin{proof}
	As in the proof of Proposition \ref{prop.destabilising}, we assume $b_2<b_1$ are the $b$-values of the intersection points of $\ell_f$ with $\partial U$. By \eqref{b}, 
	\begin{equation}\label{bounds for D-2}
	b_1 \leq  \frac{D_2.H^2}{H^3} \leq b_2+k\ .
	\end{equation}
	We know the wall $\ell$ is of slope $\frac{s}{k}$ and passes through 
	\begin{equation*}
	\Pi(v_2) = \left(\frac{D_2.H^2}{H^3},\ \frac{D_2^2.H}{2H^3} -\frac{\beta_2.H}{H^3}  \right). 
	\end{equation*}
	Since $\ell$ lies above or on the parallel line $\ell_f$, $\Pi(v_2)$ lies above or on $\ell_f$. By the classical Bogomolov inequality \eqref{BOG}, $\Pi(v_2)$ lies outside $U$. Therefore \eqref{bounds for D-2} implies that the vertical distance from $\Pi(v_2)$ to $\partial U$ which is equal to $\frac{1}{2}\left(\frac{D_2.H^2}{H^3}\right)^2 -  \left(\frac{D_2^2.H}{2H^3} - \frac{\beta_2.H}{H^3}\right)$ is maximum when $\frac{D_2.H^2}{H^3}$ is maximum and $\Pi(v_2)$ lies on $\ell_f$, see Figure \ref{blue}. 
	
	\begin{figure}[h]
		\begin{centering}
			\definecolor{zzttqq}{rgb}{0.27,0.27,0.27}
			\definecolor{qqqqff}{rgb}{0.33,0.33,0.33}
			\definecolor{uququq}{rgb}{0.25,0.25,0.25}
			\definecolor{xdxdff}{rgb}{0.66,0.66,0.66}
			
			\begin{tikzpicture}[line cap=round,line join=round,>=triangle 45,x=1.0cm,y=1.0cm]
			
			\fill[fill=blue!40!white](-3.8, 6.4) --(-3.8, 2.93)--(-2.35, 2.45)--(-2.35, 6.4) --(-3.8, 6.4);
			
			\fill[fill=blue!40!white](1.6,1.12) --(3, .67)--(3, 3.95)--(1.6 , 4) --(1.6, 1.12);
			
			\fill [fill=gray!40!white] (0,0) parabola (4,7.10) parabola [bend at end] (-4,7.10) parabola [bend at end] (0,0);
			
			\draw[->,color=black] (-5,0) -- (5,0);
			\draw  (5, 0) node [right ] {$b,\,\frac{\ch_1.H^2}{\ch_0H^3}$};
			
			\draw  (0,0) parabola (4,7.10); 
			\draw  (0,0) parabola (-4,7.10); 
			\draw  (4,7.10) node [above] {$w= \frac{b^2}{2}$};

			\draw[->,color=black] (0,-1) -- (0,7.3);
			\draw  (0, 7.3) node [above ] {$w,\,\frac{\ch_2.H}{\ch_0H^3}$};

			\draw [color=black] (3.5,.5) -- (-4,3);
			\draw [color=black] (3.5,1) -- (-4,3.5);
			
			\draw[color=black, dashed] (3,0)-- (3,4);
			\draw[color=black, dashed] (1.6,1.12) -- (1.6,0);
			\draw [dashed, color=black] (-3.8, 6.4) -- (-3.8, 0);
			\draw [dashed, color=black] (-2.35, 2.45) -- (-2.35, 0);
			
			
			\draw  (2.5, 1.33) node [above] {$\Pi(v_2)$};
			\draw  (3, 6) node [above] {\Large{$U$}};
			\draw  (1.6, 0) node [below] {$b_1$};
			\draw  (3 , 0) node [below] {$b_2+k$};
			\draw  (3.5, .5) node [right] {$\ell_f$};
			\draw  (3.5, 1) node [right] {$\ell$};
			\draw  (-3.31, 3.22) node [above] {$\Pi(v_1)$};
			\draw  (-2.35, 0) node [below] {$b_2$};
			\draw  (-3.8 , 0) node [below] {$b_1-k$};
			
			\begin{scriptsize}
			\fill (2.5, 1.33) circle (1.5pt);
			\fill (1.6,1.12) circle (1.5pt);
			\fill (1.6,0) circle (1.5pt);
			\fill (3,0) circle (1.5pt);
			\fill (3, .67) circle (1.5pt);
			\fill (3, 3.95) circle (1.5pt);
			
			\fill (-2.35, 2.45) circle (1.5pt);
			\fill (-3.8, 6.4) circle (1.5pt);
			\fill (-3.8, 0) circle (1.5pt);
			\fill (-3.8, 2.93) circle (1.5pt);
			\fill (-3.2, 3.23) circle (1.5pt);
			\fill (-2.35, 0) circle (1.5pt);

			\end{scriptsize}
			
			\end{tikzpicture}
			
			\caption{$\Pi(v_i)$ lies in the blue region.}
			
			\label{blue}
			
		\end{centering}
	\end{figure}
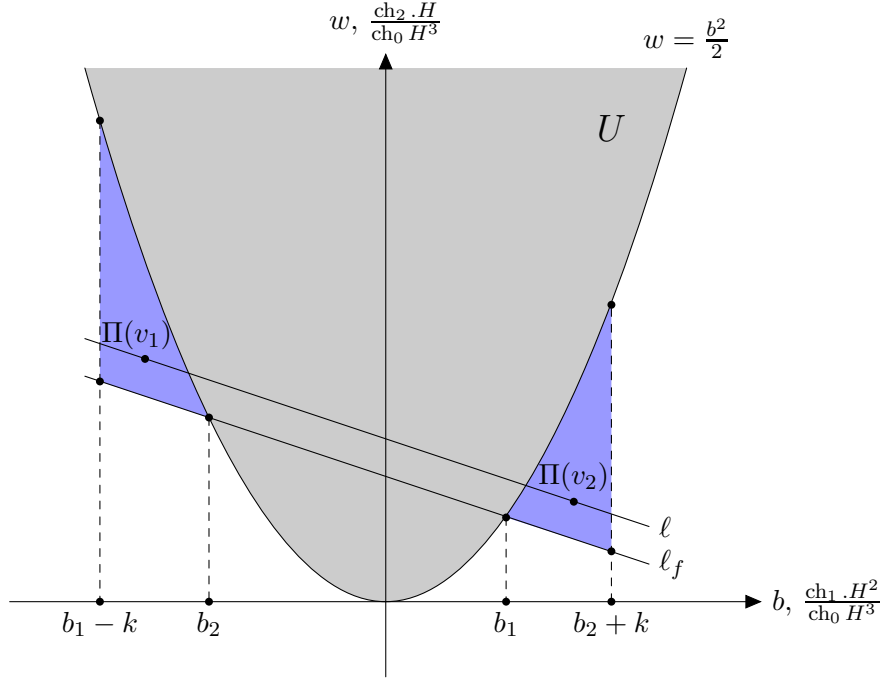
	Therefore	    
	\begin{align}\label{b-1}
	\frac{1}{2}\left(\frac{D_2.H^2}{H^3}\right)^2 -  \frac{D_2^2.H}{2H^3} + \frac{\beta_2.H}{H^3}\ & \leq \, \frac{1}{2}(b_2 +k)^2 - \left(\frac{s}{k} (b_2+k) + \frac{k^2}{8} - \frac{s^2}{2k^2} -\frac{1}{4}Q(v)\right)\nonumber\\
	& = \frac{1}{2} \left( b_2 +k - \frac{s}{k}\right)^2 - \frac{k^2}{8} +\frac{1}{4}Q(v)\nonumber\\
	& = \frac{1}{2} \left(k - \sqrt{\frac{k^2}{4} -\frac{1}{2}Q(v) }\right)^2
	-\frac{k^2}{8}+\frac{1}{4}Q(v)\nonumber\\
	& = \frac{1}{2}\left(k^2 -k \sqrt{k^2-2Q(v)}\right).  \nonumber
	\end{align}
	This proves \eqref{beta-i} for class $v_2$. Applying a similar argument proves it for class $v_1$.

	We know $E_2 \otimes D_2^{-1}$ is a torsion-free sheaf of class $(1, 0, -\beta_2, -m_2)$, thus \cite[Proposition 8.3]{feyz:thomas-noether-loci} implies 
	\begin{equation*}
	-m_2 \leq \frac{2}{3}\beta_2.H \left( \beta_2.H + \frac{1}{2H^3}\right).
	\end{equation*}
     We know $\big(E_1 \otimes \det(E_1)\big)^{\vee}[1]$ is a stable pair, so it lies in an exact triangle  
	\begin{equation*}
	E'\ \Into\, (E_1 \otimes D_1^{-1})^{\vee}[1]\, \To\hspace{-5.5mm}\To\, Q[-1],
	\end{equation*}
	with $Q$ a zero-dimensional sheaf and $E'$ a torsion-free sheaf of class $(1, 0, -\beta_1 , m_1+ \ell(Q))$ where $\ell(Q)$ is the length of $Q$. Thus \cite[Proposition 8.3]{feyz:thomas-noether-loci} implies 
	\begin{equation*}
	m_1 \leq m_1+ \ell(Q) \leq \frac{2}{3}\beta_1.H\left(\beta_1.H + \frac{1}{2H^3}\right). 
	\end{equation*}
	This completes the proof of \eqref{m-i}. 
	
	Finally we show that $(D_i, \beta_i, (-1)^{i+1}m_i) \in M(\v)$. The right hand side in \eqref{beta-i} is maximum when $Q(v)$ is maximum, so our assumption \eqref{bound for Q} implies that 
	\begin{align*}
	\frac{1}{2}\left(\frac{D_i.H^2}{H^3}\right)^2 -  \frac{D_i^2.H}{2H^3} + \frac{\beta_i.H}{H^3}\ \leq \ \frac{1}{2}\left(k^2 -k \sqrt{k^2-2Q(v)}\right) \leq\  & \frac{1}{2}\left(k^2-k(k-\frac{1}{H^3} + \frac{2}{k(H^3)^2}) \right)\\
	\ = &\ \frac{k}{2H^3} - \frac{1}{(H^3)^2}.
	\end{align*}
	By Hodge index theorem $0 \leq \frac{1}{2}\left( \frac{D_i.H^2}{H^3}\right)^2 - \frac{D_i^2H}{2H^3}$, hence 
	\begin{equation}\label{the second bound}
	\beta_iH \leq H^3\left( \frac{1}{2}\left(\frac{D_i.H^2}{H^3}\right)^2 -  \frac{D_i^2.H}{2H^3} \right) +\beta_iH \leq \frac{k}{2} - \frac{1}{H^3} < \frac{k}{2}.
	\end{equation}
	Thus \eqref{m-i} implies 
	\begin{equation*}
	(-1)^{i+1}m_i \leq  \frac{2}{3} \cdot \frac{k}{2} \left(\frac{k}{2} + \frac{1}{2H^3} \right) \leq \frac{k(k+1)}{6}
	\end{equation*}
	as claimed. 
\end{proof}

The final step to prove Theorem \ref{thm-rank zero-wall-crossing} is to show the converse of Proposition \ref{prop.destabilising}.
\begin{Prop}\label{prop.allowed classes}
	Take objects $E_1, E_2 \in \cD(X)$ of classes $v_1 = -e^{D_1}(1, 0, -\beta_1, -m_1)$ and $v_2 = e^{D_2}(1, 0, -\beta_2, -m_2)$ such that $v_1+v_2 =\v$, and 
	\begin{enumerate}
		\item $\big(E_1 \otimes \det E_1\big)^{\vee}[1]$ is a stable pair, 
		\item $E_2$ is a torsion-free sheaf, 
		\item $	\frac{1}{2}\left(\frac{D_i.H^2}{H^3}\right)^2 - \frac{D_i^2.H}{2H^3} +\frac{\beta_i.H}{H^3} \leq \frac{k}{2H^3} - \frac{1}{(H^3)^2}$ for $i=1, 2$.
	\end{enumerate}
	Then there is a point $(b, w) \in U$ such that $E_1$ and $E_2$ are $\nu_{b, w}$-stable of the same slope. In particular, their extensions are strictly $\nu_{b,w}$-semistable objects of class $\v$.  
\end{Prop}

\begin{proof}
	Taking $\ch_2$ from $v_1+v_2 =\v$ and applying Hodge index Theorem give 
	\begin{align*}
	\frac{\beta_1.H}{H^3} - \frac{D_1^2.H}{2H^3} &= \frac{\beta_2.H}{H^3} - \frac{D_2^2.H}{2H^3} + s\\
	& \geq -\frac{1}{2}\left(\frac{D_2.H^2}{H^3}\right)^2+ \frac{\beta_2.H}{H^3} + s\\
	& = -\frac{1}{2}\left(\frac{D_1.H^2}{H^3}+k\right)^2+ \frac{\beta_2.H}{H^3} + s\,.
	\end{align*}
	Since $\beta_2.H \geq 0$ by \eqref{BOG}, we get
	\begin{align}\label{l-bound}
	\frac{1}{2}\left(\frac{D_1H^2}{H^3}\right)^2 -  \frac{D_1^2H}{2H^3} + \frac{\beta_1.H}{H^3}  \geq -\frac{k^2}{2} -k \frac{D_1H^2}{H^3} +s
	\end{align}	
	The assumption in part (c) gives
	\begin{align*}
	\frac{k}{2H^3} > \frac{k}{2H^3} - \frac{1}{(H^3)^2} \geq \frac{1}{2}\left(\frac{D_1H^2}{H^3}\right)^2 -  \frac{D_1^2H}{2H^3} + \frac{\beta_1.H}{H^3}.  
	\end{align*}	
	Combining this with \eqref{l-bound} implies 
	\begin{align}\label{upper-1}
	\frac{s}{k} -\frac{k}{2}- \frac{1}{2H^3} < 	\frac{D_1.H^2}{H^3} = 	\frac{D_2.H^2}{H^3} -k\, . 
	\end{align}	
	By applying a similar argument for class $v_2$ and using part (c), one can show 
	\begin{equation*}
	\frac{D_1.H^2}{H^3}  < \frac{s}{k} -\frac{k}{2}  + \frac{1}{2H^3} \ . 
	\end{equation*} 	
	
	Let $\ell$ be the line passing through $\Pi(v_2)$ of slope $\frac{s}{k}$ (which also passes through $\Pi(v_1)$). We show that both $E_1$ and $E_2$ are $\nu_{b,w}$-stable for $(b,w) \in \ell \cap U$. Let $\ell_{v_2}$ be the line passing through
	\begin{equation*}
	\Pi(v_2) = \left(\frac{D_2.H^2}{H^3},\ \frac{D_2^2.H}{2H^3} -\frac{\beta_2.H}{H^3}  \right)
	\end{equation*}
	and the point 
	$\left(\frac{D_2.H^2}{H^3} - \frac{1}{H^3} ,\ \frac{1}{2}\left(\frac{D_2.H^2}{H^3} - \frac{1}{H^3} \right)^2  \right)$ on $\partial U$. By \cite[Lemma 8.1]{feyz:thomas-noether-loci}, there is no wall for class $v_2$ passing through the vertical line $b = \frac{D_2.H^2}{H^3} - \frac{1}{H^3}$. Thus $E_2$ is $\nu_{b,w}$-stable for any $(b,w) \in U$ above $\ell_{v_2}$ when $b < \frac{D_2.H^2}{H^3}$. 
	
	We claim the line segment $\ell \cap U$ lies above $\ell_{v_2}$. Otherwise, slope of $\ell_{v_2}$ is smaller than or equal to $\frac{s}{k}$, i.e.
	\begin{align*}
	\frac{s}{k} \geq \ &  \frac{\frac{D_2^2.H}{2H^3} -\frac{\beta_2.H}{H^3} - \frac{1}{2} \left(\frac{D_2.H^2}{H^3} - \frac{1}{H^3}\right)^2  }{\frac{1}{H^3}} \\
	= \ & \frac{-\frac{1}{2}(\frac{D_2.H^2}{H^3})^2 + \frac{D_2^2.H}{2H^3} -\frac{\beta_2.H}{H^3}  -\frac{1}{2(H^3)^2} + \frac{D_2.H^2}{(H^3)^2}  }{\frac{1}{H^3}}
	\end{align*}
	This implies 
	\begin{align*}
	\frac{1}{2}\left(\frac{D_2.H^2}{H^3}\right)^2 - \frac{D_2^2.H}{2H^3} +\frac{\beta_2.H}{H^3} \geq &\ -\frac{s}{kH^3} -\frac{1}{2(H^3)^2} + \frac{D_2.H^2}{(H^3)^2}\\
	\overset{\eqref{upper-1}}{>} & \ \frac{k}{2H^3} - \frac{1}{(H^3)^2} 
	\end{align*}
	which is not possible by part (c). Therefore $E_2$ is $\nu_{b,w}$-stable for any $(b,w) \in \ell \cap U$. By applying a similar argument, one can show the same holds for $E_1$, so any extension of $E_1$ and $E_2$ is strictly $\nu_{b,w}$-semistable for any $(b,w) \in \ell \cap U$, as claimed.
\end{proof}

We now apply the wall-crossing results in this section to prove Theorem \ref{thm-rank zero-wall-crossing}. So assume that $X$ is a Calabi-Yau 3-fold: $K_X\cong\cO_X$ and $H^1(\cO_X)=0$.
\begin{proof}[Proof of Theorem \ref{thm-rank zero-wall-crossing}]
	Part (i) follows from Lemma \ref{lem-first part of theorem 1}. For part (ii), we know in the large volume limit when $w \gg 0$, we have $\J_{b, w}(\v) = \J_{\text{ti}}(\v)$. By Lemma \ref{lem-no-strictly-semistable}, there is no strictly-tilt semistable sheaf of class $\v$, so 
	\begin{equation*}
	\J_{b, w \gg 0}(\v) = \J_{\text{ti}}(\v) = \J(\v). 
	\end{equation*}
	On the other hand, we know $\J_{b,w}(\v) = 0$ for $(b,w) \in U$ below $\ell_f$. Let $\ell$ be a wall for class $\v$ between large volume limit and $\ell_f$. Let $(b, w^{\pm})$ be points just above and below the wall $\ell$ and $(b, w_0) \in \ell$. Proposition \ref{prop.destabilising} implies that there are precisely two destabilising factors $v_1, v_2$ along the wall $\ell$. We know the slope function $\nu_{b, w}(v_1)$ is an increasing linear function with respect to $w$ and $\nu_{b, w}(v_2)$ is a decreasing linear function with respect to $w$. Also $\nu_{b, w_0}(v_1) = \nu_{b, w_0}(v_2)$, thus
	\begin{equation*}
	\nu_{b,w_+}(v_1) > \nu_{b,w_+}(v_2) \qquad \text{and} \qquad 	\nu_{b,w_-}(v_1) < \nu_{b,w_-}(v_2).
	\end{equation*}  
	This implies 
	\begin{equation*}
	U(v_1, v_2, (b, w_+), (b, w_-)) = 1 \qquad\ \text{and} \qquad  	U(v_2,v_1, (b, w_+), (b, w_-)) = -1
	\end{equation*}
	The wall-crossing formula \eqref{wcf} gives
	\begin{equation*}
	J_{b, w_-}(\v) = \J_{b, w_+}(\v) + \sum_{
\substack{
v_1= -e^{D_1}(1, 0, \beta_1, m_1)\\
v_2= e^{D_2}(1, 0, \beta_2, m_2)\\
v_1+v_2 =\v\\
\Pi(v_i) \in \ell \ \text{for $i=1, 2$}
}	
}(-1)^{\chi(v_1, v_2) +1}\chi(v_1 ,v_2)\J_{b, w_+}(v_1)\J_{b,w_+}(v_2).
	\end{equation*}    
	By Proposition \ref{prop.destabilising}, $\J_{b,w_+}(v_i) = \J_{b,w \gg 0}(v_i)$. We may assume $\mu_H(E_1) < b < \mu_H(E_2)$. Thus Lemma \ref{lem.large volume limit} shows that 
	$$
	\J_{b, w \gg 0}(v_2) = \J(v_2) = \J(1, 0, -\beta_2, -m_2)
	$$ 
	The latter counts torsion-free sheaves of class $(1, 0, -\beta_1, -m_1)$ which are ideal sheaves of 1-dimensional subscheme after tensioning by a line bundle $L$ with torsion $c_1$. Thus 
	\begin{equation*}
	\J_{b,w_+}(v_2) = \J(1, 0, -\beta_2, -m_2) = \big(\#H^2(X,\Z)_{\mathrm{tors}}\big) I_{m_2, \beta_2}. 
	\end{equation*}
	Similarly, we know $J_{b,w_+}(v_1) = J_{\infty}(-1, 0, \beta_1, m_1)$ where the latter counts dual of stable pairs of class $(-1, 0, \beta_1, -m_1)$ up to tensoring by a line bundle with torsion $c_1$-class, thus
	\begin{equation*}
	J_{b,w_+}(v_1) = \big(\#H^2(X,\Z)_{\mathrm{tors}}\big) P_{-m_1, \beta_1}.  
	\end{equation*}  
	Summing up over all walls for class $\v$ between the large volume limit and $\ell_f$, and applying Proposition \ref{prop.calsses of factors} and Proposition \ref{prop.allowed classes} imply the final wall-crossing formula in Theorem \ref{thm-rank zero-wall-crossing}.  
\end{proof}

\begin{Lem}\label{lem.quintic}
	Theorem \ref{thm-rank zero-wall-crossing} holds if $X$ is a smooth projective quintic threefold or a smooth projective threefold of complete intersection of quadratic and quartic hypersurfaces in $\PP^5$. 
\end{Lem}
\begin{proof}
	By \cite{chunyi:stability-condition-quintic-threefold, liu:bg-ineqaulity-quadratic}, Conjecture \ref{conjecture} holds on $X$ for $(b,w)$ satisfying 
	\begin{equation}\label{in for b, w}
	w\ >\ \frac{1}{2} b^2 + \frac{1}{2}\big(b - \lfloor b \rfloor\big)\big (\lfloor b \rfloor - b +1\big). 
	\end{equation}
	Hence if $b \in \mathbb{Z}$, it holds for any $w > \frac{1}{2}b^2$. To prove Theorem \ref{thm-rank zero-wall-crossing}, we applied Conjecture \ref{conjecture} in two places: (1) to find the line $\ell_f$ in \eqref{bg-conjecture-our use} and (2) in the proof of \cite[Proposition 8.3]{feyz:thomas-noether-loci}. It has been shown in \cite[Theorem 3.2]{feyz:thomas-noether-loci} that (2) holds true if Conjecture \ref{conjecture} is true for $(b,w)$ satisfying \eqref{in for b, w}. 
	
	For case (1), since $\Pic(X) = \mathbb{Z}.H$, we know $k \in \Z$. As shown in \eqref{lower}, the line $\ell_f$ intersects $\partial U$ at two points with $b_1-b_2 \geq 1$ if $k \geq 2$, so $\ell_f$ intersects a vertical line $b =b_0 \in \mathbb{Z}$ at a point in the closure $\overline{U}$ as we required. If $k=1$, then $Q(v) \in \frac{1}{2}\mathbb{Z}$, so the bound \eqref{equiavalent bound for Q} implies that $Q(v) \leq 0$. Thus \eqref{lower} implies again $b_1-b_2 \geq 1$.  
\end{proof}

\section{OSV conjecture}\label{section.osv}

In this section, we prove a slight modification of Toda's formulation of OSV conjecture \cite[Conjecture 1.1]{toda:bogomolov-counting}. Part (i) of \cite[Conjecture 1.1]{toda:bogomolov-counting} follows from Theorem \ref{thm-rank zero-wall-crossing}(i). Thus we only consider the second part of \cite[Conjecture 1.1]{toda:bogomolov-counting}.


Fix $k >0$. We consider the invariants $\J(0, kH, \beta, m)$ which counts $H$-Gieseker semistable sheaves of Chern character $(0, kH, \beta, m)$\footnote{We replaced the class $(0, nH, -\beta, -m)$ in Toda's notations \cite{toda:bogomolov-counting} by $(0, kH, \beta, m)$.}. The generating series of these invariants is defined by 
\begin{equation*}
\mathcal{Z}^{k}_{D4}(x, y) \coloneqq \sum_{\beta, m} \ \J(0, kH, \beta, m)\,x^{-m}y^{-\beta}\ .  
\end{equation*}
  Define the subset
  \begin{equation*}
  N(k) \coloneqq  \left\{ D \in H^2(X,\mathbb{Z}) \colon 	\left(\frac{D.H^2}{H^3}\right)^2 - \frac{D^2.H}{H^3}\, <\, \frac{k}{2H^3}   \right\} \subset H^2(X, \mathbb{Z}). 
  \end{equation*}
  If $\Pic(X) = \mathbb{Z}.H$, for instance, then $N(k) =H^2(X, \mathbb{Z})$. For any $\epsilon >0$, consider the generating series
\begin{align*}
&\mathcal{Z}^{k, \epsilon}_{D6-\overline{D6}}(x, y , z) \coloneqq\\ &\sum_{\substack{D_1, D_2 \,\in\, N(k)\\D_2-D_1=kH}}
x^{\frac{D_1^3}{6} - \frac{D_2^3}{6}}\, y ^{\frac{D_1^2}{2}-\frac{D_2^2}{2}}\,z^{\frac{k^3H^3}{6} + \frac{kH.c_2(X)}{12}} 
\qe^{k, \epsilon}(xz^{-1}, x^{D_2}yz^{-kH})\p^{k, \epsilon}(xz^{-1}, x^{-D_1}y^{-1}z^{-kH})
\end{align*}
where 
\begin{equation*}
\p^{k, \epsilon} (x, y) \coloneqq \sum_{(\beta, -m)\, \in\, C(k, \epsilon)} \p_{m,, \beta}\; x^{m}y^{\beta}\ , \qquad \qe^{k, \epsilon} (x, y) \coloneqq \sum_{(\beta, -m) \,\in\, C(k, \epsilon)} \qe_{m,\, \beta} \;x^{m}y^{\beta}.
\end{equation*}
and
\begin{equation}\label{C}
C(k, \epsilon) \coloneqq \left\{(\beta, m) \in H^4(X, \Z) \oplus H^6(X, \Z) \colon \beta.H < \epsilon k^2\ ,\  m < \epsilon k^3  \right\}. 
\end{equation}
 A slight modification of \cite[Conjecture 1.1(ii)]{toda:bogomolov-counting} says the following:
\begin{Thm}\label{thm-toda-osv}
	Let $X$ be a smooth projective Calabi-Yau 3-fold. 
	For any $\xi \geq 1$, there are $\mu >0$, $\delta>0$ and a constant $k(\xi, \mu) >0$ which depends only on $\xi$, $\mu$ such that for any $k> k(\xi, \mu)$, we have the equality of the generating series, 
	\begin{equation}\label{derivative}
	\mathcal{Z}_{D4}^k(x,y)= \big(\#H^2(X,\Z)_{\mathrm{tors}}\big)^2\left.\frac{\partial}{\partial z}\mathcal{Z}^{\ k, \,\epsilon = \frac{\delta}{k^{\xi}}}_{D6-\overline{D6}}(x, y, z)\right|_{z=-1} 
	\end{equation}
	modulo terms of $x^{m}y^{\beta}$ with 
	\begin{equation}\label{terms}
	-\frac{H^3}{24}k^3\left(1- \frac{\mu}{k^{\xi}}\right) \leq m + \frac{(\beta.H)^2}{2kH^3}.
	\end{equation}
\end{Thm}
\begin{Rem}\label{rem-toda}
	There are only three places in the above set-up that are different from the Toda ones \cite{toda:bogomolov-counting}:
	\begin{enumerate}
		\item[(i)] In the definition of generating series $\mathcal{Z}^{k, \epsilon}_{D6-\overline{D6}}(x, y , z)$, we added the extra condition that $D_1, D_2 \in N(k)$.
		\item[(ii)] In the definition of $C(k, \epsilon)$, Toda restricts the second factor $m$ to satisfy $\abs{m} < \epsilon k^3$. 
			\item [(iii)] We added the extra factor $\left(\#H^2(X,\Z)_{\mathrm{tors}}\right)^2$ in \eqref{derivative}.
	\end{enumerate}     
\end{Rem} 
\begin{proof}[Proof of Theorem \ref{thm-toda-osv}]
	By applying a similar argument as in \cite[Section 3.8]{toda:bogomolov-counting}, one can obtain the coefficient of $x^{-m}y^{-\beta}$ in the right hand side of \eqref{derivative} and show that the equality \eqref{derivative} is equivalent to 
	\begin{align}\label{j}
	&\J(0, kH, \beta, m) = \nonumber\\
	 &\left(\#H^2(X,\Z)_{\mathrm{tors}}\right)^2\sum_{\substack{v_1\, =\, -e^{D_1}(1, 0, -\beta_1, -m_1)\\ v_2 \,=\, e^{D_2}(1, 0, - \beta_2, -m_2)\\
			v_1+v_2\, =\, (0,\, kH,\, \beta,\, m) \\
			(\beta_i,\ (-1)^{i+1}m_i) \,\in\, C\left(k, \epsilon \right)\\
			D_1, D_2\, \in\, N(k)
	}}  (-1)^{\chi(v_2, v_1) -1}\;\chi(v_2, v_1)\qe_{m_2, \beta_2}\,\p_{-m_1, \beta_1}\, .
	\end{align}
	By our assumption \eqref{terms}, we only need to consider the terms $x^{-m}y^{-\beta}$ with 
	\begin{equation*}
	\frac{H^3}{24}k^3\left(\frac{\mu}{k^{\xi}}\right) > -m + \frac{(\beta.H)^2}{2kH^3} + \frac{k^3H^3}{24} = \frac{kH^3}{12}Q(\v)
	\end{equation*}
	for $\v \coloneqq (0, kH, \beta, m)$, i.e.
	\begin{equation}\label{imp}
	Q(\v) < \frac{k^2}{2}\, \frac{\mu}{k^{\xi}}\,.
	\end{equation}
	For any $\xi \geq 1$, there are $\mu >0$, $\delta >0$ and $k(\xi, \mu) >0$ such that for $k > k(\xi, \mu)$ the following three conditions are satisfied: 
	\begin{equation}\label{cond}
	\text{(i) }\mu < \frac{2}{H^3}\delta \ , \text{\qquad (ii) }\delta < \frac{1}{4} -\frac{1}{kH^3} \ , \qquad\text{(iii) }\frac{\mu}{k^{\xi}} \ \leq\ 1 - \left(1 - \frac{1}{kH^3} + \frac{2}{k^2(H^3)^2} \right)^2. 
	\end{equation}	
%
%
	We also set $\epsilon \coloneqq \frac{\delta}{k^{\xi}}$. Combining \eqref{imp} and \eqref{cond}(iii) imply that the condition \eqref{bound for Q} holds, so we may apply Theorem \ref{thm-rank zero-wall-crossing} and the results in Section \ref{section.wall} for class $\v$. To prove \eqref{j}, we need to show the following two claims: 
	\begin{enumerate}
		\item Take $E_1, E_2 \in \cD(X)$ of classes $v_1, v_2$ as in \eqref{j} such that $E_1 \otimes D_1^{-1}$ is a torsion-free sheaf and $(E_2 \otimes D_2^{-1})^{\vee}$ is a stable pair, then there is a point $(b,w) \in U$ where $E_1$ and $E_2$ are both $\nu_{b,w}$-stable of the same slope, and so they make a wall for the class $\v = (0, kH, \beta, m)$. 
		\item All classes $v_1, v_2 \in K(X)$ which give a non-zero term in the wall-crossing formula for $\J(v)$ in Theorem \ref{thm-rank zero-wall-crossing} are included in \eqref{j}.    
	\end{enumerate}
    In \eqref{j}, we know $D_i \in N(k)$ for $i=1, 2$, so 
	\begin{equation*}
	\frac{1}{2}\left(\frac{D_i.H^2}{H^3}\right)^2 - \frac{D_i^2.H}{2H^3}\, <\, \frac{k}{4H^3}. 
	\end{equation*}
	Since $(\beta_i,\ (-1)^{i+1}m_i) \in C(k, \epsilon)$, we get 
	\begin{equation*}
	\frac{1}{2}\left(\frac{D_i.H^2}{H^3}\right)^2 - \frac{D_i^2.H}{2H^3} \ +\ \frac{\beta_iH}{H^3}\ <\, \frac{k}{4H^3} + \frac{1}{H^3}\epsilon k^2 < \frac{k}{2H^3} -\frac{1}{(H^3)^2}\,.
	\end{equation*}
	Here the last inequality follows from the assumptions $\xi \geq 1$ and \eqref{cond}(ii), i.e. 
	\begin{equation*}
	\frac{1}{H^3}\epsilon k^2 = \frac{k^2}{H^3}\frac{\delta}{k^{\xi}} < \frac{k}{H^3}\delta < \frac{k}{4H^3} - \frac{1}{(H^3)^2}
	\end{equation*}
	Thus Proposition \ref{prop.allowed classes} implies claim (a). 
	
	To prove claim (b), take two classes $(D_i, \beta_i, (-1)^{i+1}m_i) \in M(\v)$ for $i=1, 2$ satisfying inequalities \eqref{beta-i} and \eqref{m-i} in Proposition \ref{prop.calsses of factors}, then we only need to show $D_i \in N(k)$ and $(\beta_i, (-1)^{i+1}m_i) \in C(k , \epsilon)$. 
	We know 
	\begin{equation*}
\frac{k^2}{2}\left(1 -\sqrt{1-2\frac{Q(\v)}{k^2}}\right) = \frac{k^2}{2}\frac{2\frac{Q(\v)}{k^2}}{1+ \sqrt{1-2\frac{Q(\v)}{k^2}} } \leq Q(\v) \overset{\eqref{imp}}{<}  \frac{k^2}{2}\, \frac{\mu}{k^{\xi}} \overset{\eqref{cond}(i)}{<} \frac{k^2}{H^3}\frac{\delta}{k^{\xi}} = \frac{k^2\epsilon}{H^3}.
\end{equation*} 
	Combining this with \eqref{beta-i} implies 
	\begin{equation*}
	\frac{1}{2}\left(\frac{D_iH^2}{H^3}\right)^2 -  \frac{D_i^2H}{2H^3} + \frac{\beta_i.H}{H^3} < \frac{k^2}{H^3}\epsilon.
	\end{equation*}
	By Hodge index Theorem and \eqref{BOG}, both terms $\frac{1}{2}\left(\frac{D_iH^2}{H^3}\right)^2 -  \frac{D_i^2H}{2H^3}$ and $\beta_i.H$ are non-negative, so the above gives $\beta_i.H \leq \epsilon k^2$ and 
	\begin{equation*}
	\left(\frac{D_iH^2}{H^3}\right)^2 -  \frac{D_i^2H}{H^3} < 2\frac{k^2}{H^3}\epsilon = \frac{2k^2}{H^3}\frac{\delta}{k^{\xi}} \overset{\eqref{cond}(ii)}{<} \frac{2k}{H^3}\left(\frac{1}{4} -\frac{1}{kH^3} \right) < \frac{k}{2H^3} 
	\end{equation*}
	which proves $D_i \in N(k)$.
	Moreover \eqref{m-i} gives
	\begin{equation}
	(-1)^{i+1}m_i \leq\  \frac{2}{3}\beta_i.H \left( \beta_i.H + \frac{1}{2H^3}\right) \leq \frac{2}{3}\epsilon k^2 \left(\epsilon k^2 + \frac{1}{2H^3} \right). 
	\end{equation}
	To finish the proof of claim (b), it suffices to show
	\begin{equation*}
	\frac{2}{3}\epsilon k^2 \left(\epsilon k^2 + \frac{1}{2H^3} \right) \leq  \epsilon k^3.
	\end{equation*}
	This is equivalent to $\epsilon = \frac{\delta}{k^{\xi}} \leq \frac{3}{2k} - \frac{1}{2k^2H^3}$, so it is enough to show   
	\begin{equation*}
	\frac{\delta}{k^{\xi}} \leq \frac{3}{2k} - \frac{1}{2k} = \frac{1}{k}.  
	\end{equation*}
	which clearly holds by \eqref{cond}(ii).

\end{proof}


\section{Joyce--Song pair}\label{js section}
As before, we fix a rank 0 class $\v = (0, D, \beta, m )\in K(X)$ with
\beq{vdef}
\ch\_H(\v)\=(0,\,k,\,s,\,d),
\eeq
where $k>0$. Then we pick $n\gg0$ and set $\v_n:=\v-[\cO_X(-n)]$, so
\begin{equation}\label{class vn}
\ch\_H(\v_n)\=\Big(\!-1, \ k+n,\ s-\tfrac12n^2 , \ d +\tfrac16n^3\Big).
\end{equation}
In this section, we do wall-crossing for class $\v_n$ instead of class $\v$, and capture the information of slope-semistable sheaves of class $\v$ along the Joyce-Song wall where $\v$ and $\cO_X(-n)[1]$ have the same slope.
\bigskip

\noindent We know any wall for class $\v_n$ is a line segment $\ell \cap U$ where $\ell$ is a line passes through
\begin{equation*}
\Pi(\v_n) = \left(-k-n\ , \ -s+ \frac{1}{2}n^2  \right). 
\end{equation*}

\begin{figure}[h]
	\begin{centering}
		\definecolor{zzttqq}{rgb}{0.27,0.27,0.27}
		\definecolor{qqqqff}{rgb}{0.33,0.33,0.33}
		\definecolor{uququq}{rgb}{0.25,0.25,0.25}
		\definecolor{xdxdff}{rgb}{0.66,0.66,0.66}
		
		\begin{tikzpicture}[line cap=round,line join=round,>=triangle 45,x=1cm,y=0.9cm]
		
		\draw[->,color=black] (-4,0) -- (4,0);
		\draw  (4, 0) node [right ] {$b,\,\frac{\ch_1\!.\;H^2}{\ch_0H^3}$};
		
		\fill [fill=gray!25!white] (0,0) parabola (3,4) parabola [bend at end] (-3,4) parabola [bend at end] (0,0);
		
		
		\draw  (0,0) parabola (3.1,4.27); 
		\draw  (0,0) parabola (-3.1,4.27); 
		\draw  (3.8 , 3.6) node [above] {$w= \frac{b^2}{2}$};
		
		\draw[->,color=black] (0,-.8) -- (0,4.7);
		\draw  (1, 4.1) node [above ] {$w,\,\frac{\ch_2\!.\;H}{\ch_0H^3}$};
		
		\draw [color=black] (-2.8,1.95) -- (2.5, .5);
		\draw [color=black] (-2.8,1.95) -- (3.5, 3.5);
		\draw [dashed] (-1.96, 1.75) -- (-1.96,-0.15);		
		\draw [dashed] (1.35,.8) -- (1.35,-0.1);
		
		\draw [dashed] (-2.8,1.95) -- (-2.8,4.2);
		
		\draw  (-1.95, 0) node[below]{$b_2^f$};
		\draw  (1.35, 0) node[below]{$b_1^f$};
		\draw  (-2.8, 1.6) node {{\small$\Pi(\v_n)$}};
		\draw  (-2.3,2.05) node[above] {{\small$\Pi(\w)$}};
		\draw  (2, .35) node {$\ell_f$};
		\draw  (3.3, 2.9) node [above] {$\ell$};
		\draw  (2, 3.2) node [above] {\Large{$U$}};
		\fill (-2.4,2.05) circle (2pt);
		\fill (-2.8,1.95) circle (2pt);
		
		\end{tikzpicture}
		
		\caption{The line $\ell_f$}
		\label{figUwm}
	\end{centering}
\end{figure}
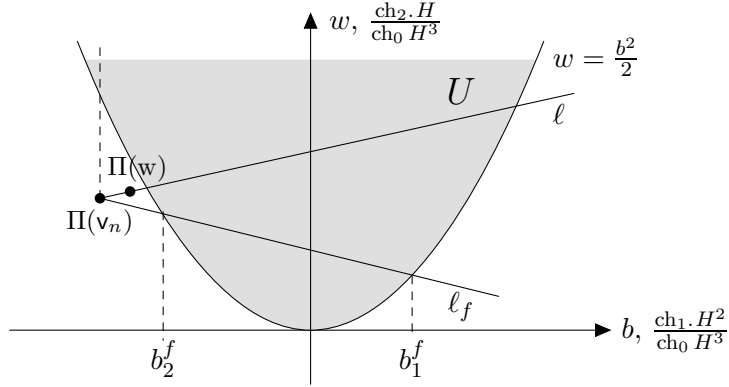

\noindent If $b > -k-n$ and $w \gg 0$, any $\nu_{b,w}$-semistable object $E \in \cA(b)$ of class $\v_n$ is isomorphic to the derived dual of a stable pair up to twisting by a line bundle, see Lemma \ref{lem.large volume limit}. On the other hand, conjectural Bogomolov inequality \eqref{quadratic form} implies that there is a line $\ell_f$ described in \cite[Equation (23)]{feyz:rank-r-dt-theory-from-1} such that there is no $\nu_{b,w}$-semistable object of class $v_{n}$ for $(b,w)$ below $\ell_f$, see Figure \ref{figUwm}. Hence there are finitely many walls for class $\v_n$ between the large volume limit and $\ell_f$, see \cite[Proposition 1.4]{feyz:rank-r-dt-theory-from-1} for more details. 


Let $M_{\v, n}$ be the set of all rank $-1$ classes $\al = (-1, \ch_1, \ch_2, \ch_3) \in K(X)$ with non-negative discriminant $\Delta_H(\al) \geq 0$ such that 
\begin{equation}\label{the set M}
n +k < \frac{\ch_1H^2}{H^3} \leq n - \frac{k}{3} \qquad \text{and} \qquad \text{$\Pi(\al)$ lies above or on $\ell_f$}. 
\end{equation}
For any class $\al \in M_{\v , n}$ we define a line $L_{\al} \subset \R^2$ by using induction on $\Delta_H(\al)$: 
\begin{itemize*}
	\item If $\Delta_H(\al) = 0$, then $L_{\al}$ is the vertical line passing through $\Pi(\al)$ (which is of gradient $-\infty$).
	\item If $\Delta_H(\al) >0$, then $L_{\al}$ is the line of smallest gradient passing through $\Pi(\al)$ which satisfies the following condition:
	
	If an object $E$ of class $\al$ gets destabilised along a wall $\ell$ on the right of $\Pi(\al)$\footnote{In other words, $\ell \cap U$ lies to the right of the vertical line passing through $\Pi(E)$.} and above $L_{\al}$ with the destabilising sequence $E_1 \rightarrow E \rightarrow E_2$, then 
	\begin{enumerate*}
		\item one of the factors $E_i$ is of rank $-1$ with $\ch(E_i) \in M_{\v, n}$ and $\Delta_H(E_i) < \Delta_H(\al)$. Moreover $\ell$ lies above $L_{[E_i]}$ (i.e. the gradient of $\ell$ is bigger than $L_{[E_i]}$), 
		\item the other factor $E_j$ is of rank zero with $\frac{\ch_1(E_j)H^2}{H^3} < k $ and there is no wall for class $[E_j]$ above or on $\ell$. 
	\end{enumerate*} 
\end{itemize*}

%
We know for any class $\al \in M_{\v , n}$, there is $w_0 >0$ such that there is no wall for $\al$ crossing the vertical line $b =0$ at a point $(0, w)$ with $w > w_0$. To find $L_{\al}$, we may start with the line $L$ passing through $\Pi(\al)$ and $(0, w_0)$. Since there is no wall for class $\al$ above $L$, the above condition is satisfied. Then we rotate the line $L$ clockwise about $\Pi(\al)$ as far as both (a) and (b) hold true, and the limiting line will be $L_{\al}$.  


\begin{Thm}{\cite[Section 2]{feyz:rank-r-dt-theory-from-1}}
Suppose an object $E$ of class $\v_n$ gets destabilised along a wall $\ell$ with a destabilising sequence $E_1 \rightarrow E \rightarrow E_2$, then  
\begin{itemize*}
	\item one of the factors $E_i$ is of rank $-1$ with $\ch(E_i) \in M_{\v , n}$ and the wall $\ell$ lies above $L_{[E_i]}$, 
	\item the other factor $E_j$ is of rank zero with $\frac{\ch_1(E_j)H^2}{H^3} \leq k$ such that there is no wall for class $[E_j]$ above or on $\ell_f$. 
\end{itemize*} 
Moreover $\ch_1(E_j)H = kH^3$ if and only if the wall $\ell$ is the \emph{Joyce--Song} wall, i.e. $E_1$ is a slope-semistable sheaf of class $\v$ and $E_2 = T(-nH)[1]$ for a line bundle $T$ with torsion first Chern class.	
\end{Thm}
\begin{proof}
	By \cite[Lemma 2.1 \& Lemma 2.2]{feyz:rank-r-dt-theory-from-1}\footnote{Here we replaced the parameters $n_0, c$ and $s_0$ in \cite[Section 2]{feyz:rank-r-dt-theory-from-1} by $n, k$ and $s$, respectively.}, one of the factors $E_j$ is of rank zero and there is no wall for $[E_j]$ above or on $\ell_f$. The other factor $E_i$ is of rank $-1$ and $\ch(E_i) \in M_{\v , n}$. \cite[Proposition 2.6]{feyz:rank-r-dt-theory-from-1} implies that $E_i$ is either
	\begin{enumerate}
		\item[(a)] close to $v_{n}$ (see \cite[Definition 2.3]{feyz:rank-r-dt-theory-from-1}), or 
		\item[(b)] the wall $\ell$ lies in the safe area of $[E_i]$ (see \cite[Definition 2.7]{feyz:rank-r-dt-theory-from-1}).   
	\end{enumerate}
  In case (b), by induction on $\Delta_H(E_i)$, one sees $L_{[E_i]}$ lies below the safe area of $[E_i]$, see \cite[Definition 2.7 and Lemma 2.8]{feyz:rank-r-dt-theory-from-1}. Thus the wall $\ell$ lies above $L_{[E_i]}$ as claimed.    
  
  In case (a), for any class $\al$ close to $\v_n$, the line $\ell(\al)$ is defined to be the line parallel to $\ell_f$ through the point $\Pi(\al)$. 
  We can again proceed by induction on $\Delta_H(E_i)$ and apply \cite[Proposition 2.6]{feyz:rank-r-dt-theory-from-1} to prove the line $\ell([E_i])$ is above $L_{[E_i]}$. Since the wall $\ell$ lies above $\ell([E_i])$, we get $\ell$ lies above $L_{[E_i]}$ as well.



   When $\ch_1(E_j)H = kH^3$, the final claim follows from \cite[Proposition 2.5]{feyz:rank-r-dt-theory-from-1}.    
\end{proof}

\subsection*{Wall-crossing formula} 
From now on, we assume $X$ is a Calabi-Yau 3-fold with $\Pic(X) = \Z.H$. In this case, we can simplify the wall-crossing formula \eqref{wcf} for any class in $M_{\v , n}$ and so for $v_{n}$. 


\begin{Lem}\label{lem-rank0-tilt-j}
	For any rank zero class $\alpha \in K(X)$ with $\ch_1(\alpha) \neq 0$, we have 
	\begin{equation*}
	\J(\al) = \J_{b,w}(\alpha) \qquad \text{for any $b \in \mathbb{R}$ and $w \gg 0$}. 
	\end{equation*}
\end{Lem}
\begin{proof}
	By Lemma \ref{lem.large volume limit}, we know $\J_{b, w \gg 0}(\al) = \J_{\text{ti}}(\al)$. So we only need to analyse the wall-crossing formula \eqref{wcf-tilt-gieseker} for class $\al$. Consider the semistable factors $\alpha_1, \dots, \al_q$. Then $U(\al_1, \dots, \al_k; \text{Gi}, \text{ti}) = 0$ unless $\widetilde{p}(\al_i)(t) = \widetilde{p}(\al_i)(t)$, i.e. $\nu_H(\al_i) = \nu_H(\al)$. But then 
	\begin{equation*}
	\chi(\al_i, \al_j) = \ch_2(\al_i)\ch_1(\al_j) -\ch_1(\al_i)\ch_2(\al_j) = 0
	\end{equation*}    
	as $\ch_1(\al_i) = k_iH$ for some $k_i \in \mathbb{Z}$. Thus the only non-zero term in \eqref{wcf-tilt-gieseker} is for $q=1$, so $\J_{\text{ti}}(\al) = \J(\al)$ as claimed.  
\end{proof}
 
\begin{Prop}\label{prop-wall-crossing-1}
	Fix $\al \in M_{\v, n}$ and let $\ell$ be a wall for class $\al$ which lies above $L_{\al}$. Suppose  $(b\_0, w_0^+)$ and $(b\_0, w_0^-)$ are points just above and below the wall $\ell$, then
\begin{align}\label{wcf-prop}
	\J_{b\_0, w_0^-}(\al) = \sum_{
\substack{
q \geq 1,\ \al_1, \dots , \al_q \,\in\, \ell\\	
\al_1 \in M_{\v, n}\\
\al_i = (0,\, k_iH,\, \beta_i,\, m_i) \, \text{for $i \in [2, q]$} \\
0 < k_i < k\\
\al_1+ \al_2+ \dots + \al_q = \al 
}
}\frac{1}{(q-1)!}\J_{b\_0, w_0^+}(\al_1) \prod_{i = 2}^q(-1)^{\chi(\al_i, \al)}\chi(\al_i, \al) \J(\al_i)\,. 
\end{align}
Here $\al_i \in \ell$ means $\Pi(\al_i)$ lies on the line $\ell$ if $i =1$, and if $2 \leq i \leq q$, the slope $\frac{\beta_iH}{k_iH^3}$ is equal to the gradient of $\ell$. 
\end{Prop}
\begin{proof}
The argument is similar to \cite[Theorem 5.8]{toda:curve-counting-theories-visa-stable-objects-i}. In the wall-crossing formula \eqref{wcf}, we know the coefficient $U(\alpha_1, \dots , \alpha_q, (b\_0, w_0^+), (b\_0, w_0^-))$  is zero unless 
\begin{equation*}
\nu\_{b\_0, w\_0}(\al_1) = \dots = \nu\_{b\_0, w\_0}(\al_q)  
\end{equation*}	
where $(b\_0, w\_0)$ lies on the wall $\ell$. Since $\ell$ lies above $L_{\al}$, by definition one of the factors $\al_{e}$ (lying in position $e$) is of rank $-1$ and $\al_e \in M_{\v, n}$. The other factors $\al_i$ for $i \neq e$ are of rank zero with $\ch_1(\al_i)H^2 < kH^3$. The factors of rank zero have the same $\nu\_H$-slope \eqref{nuslope}, so they have the same $\nu_{b,w}$-slope with respect to any $(b,w) \in U$.

\textbf{Step 1.} We claim 
\begin{equation}\label{S}
S\left(\al_1, \dots, \al_{q} ;\, (b_0, w_0^+), (b_0, w_0^-)\right) = 
\begin{cases}
(-1)^{e-1} &\quad\text{if $e= q-1$ or $q$,}  \\
0 &\quad\text{otherwise.} \\ 
\end{cases}
\end{equation}

\noindent First assume $q >2$ and $e \leq q -2$. If $S\left(\al_1, \dots, \al_{q} ;\, (b\_0, w_0^+), (b\_0, w_0^-)\right) \neq 0$, then for $i= q -1$, we have $\nu_{b\_0, w_0^+}(\al_i) = \nu_{b\_0, w_0^+}(\al_{i+1})$, so we must have 
\begin{equation}\label{con}
\nu_{b\_0, w_0^-}(\al_1+ ... + \al_{q -1}) > \nu_{b\_0, w_0^-}(\al_{q}). 
\end{equation}
Let $\al_e = (-1, k_eH, \beta_e , m_e)$ and $\al_i = (0, k_iH, \beta_i , m_i)$ for $i \neq e$, then
\begin{equation*}
\frac{\beta_iH}{k_iH^3} = \nu\_{b\_0, w\_0}(\alpha_i)\ =\ \nu\_{b\_0, w\_0}(\alpha_e) = \frac{\beta_eH +w\_0H^3}{k_eH^3 +b\_0H^3}\,. 
\end{equation*} 
To have a $\nu\_{b\_0, w\_0}$-semistable object of class $\al_{e}$ in $\cA_{b_0}$, we must have $k_eH^3 +b\_0H^3 > 0$. Thus for $w_0^- < w\_0$, we get $\nu_{b\_0, w_0^-}(\al_e) < \nu_{b\_0, w_0^-}(\al_i)$ for $i\neq e$, so   
\begin{equation*}
\nu_{b\_0, w_0^-}(\al_e) = \frac{\beta_eH +w_0^-H^3}{k_eH^3 +b_0H^3} < \frac{\beta_eH +w_0^-H^3+ \sum_{i \neq e, q} \beta_iH }{k_eH^3 +b_0H^3+\sum_{i \neq e, q} k_iH^3} < \frac{\sum_{i \neq e, q} \beta_iH }{\sum_{i \neq e, q} k_iH^3} = \nu_{b\_0, w_0^-}(\al_i).
\end{equation*}
The middle term is equal to $\nu_{b\_0, w_0^-}(\al_1+ ... + \al_{q -1})$ and the right hand one is equal to $\nu_{b_0, w_0^-}(\alpha_q)$, so this is in contradiction to \eqref{con}. Thus $S\left(\al_1, \dots, \al_{q} , (b\_0, w_0^+), (b\_0, w_0^-)\right) = 0$.

The same argument as above shows that for $i =1, \dots, e-1$,
\begin{equation*}
\nu\_{b_0, w_0^+}(\alpha_i) = \nu\_{b_0, w_0^+}(\alpha_{i+1}) \qquad \text{and} \qquad
\nu\_{b_0, w_0^-}(\alpha_1 + \dots + \alpha_i) > \nu\_{b_0, w_0^-}(\alpha_{i+1} + \dots + \alpha_q), 
\end{equation*}  	
and if $e = q-1$, we have 
\begin{equation*}
\nu\_{b_0, w_0^+}(\alpha_e) > \nu\_{b_0, w_0^+}(\alpha_{e+1}) \qquad \text{and} \qquad
\nu\_{b_0, w_0^-}(\alpha_1 + \dots + \alpha_e) < \nu\_{b_0, w_0^-}(\alpha_e). 
\end{equation*}
Thus $S\left(\al_1, \dots, \al_{q} ;\, (b_0, w_0^+), (b, w_0^-)\right) = (-1)^{e-1}$ if $e= q$, or $q-1$.   

\bigskip
	
\textbf{Step 2.} The next claim is  
	\begin{equation}\label{formula for u}
	U(\al_1, \dots , \al_{q} ;\, (b\_0, w_0^+), (b\_0, w_0^-)) = \frac{(-1)^{e-1}}{(e-1)!(q -e)!}\,. 
	\end{equation}
	Here $e$ is again the position of the factor of rank $-1$. Note that the slope function $\nu_{b_0, w}$ of any class is a linear function of $w$. The same computations as in Step 1 implies that if for some $i < j$, 
	\begin{equation*}
	\nu_{b_0, w_0^-}(\al) = \nu_{b_0, w_0^-}(\al_i +\al_{i+1} + \dots + \al_j)
	\end{equation*}
	then $i=1$ and $j=q$. Thus the parameter $p$ in the definition \eqref{u} must be equal to $1$ 
	and we only need to pick $t \in [1, q]$ and $0 =a_0 < a_1 ...< a_t =q$. Define
	\begin{equation*}
	\mathcal{B}_i \coloneqq \al_{a_{i-1} +1} + \dots +\al_{a_i} \qquad \text{for $i \in [1, t]$}.
	\end{equation*}
	We require $\nu_{b_0, w_0^+}(\mathcal{B}_i) = \nu_{b_0, w_0^+}(\al_j)$ for $a_{i-1}< j \leq a_i$. Thus for any $i \in [1, t]$, either 
	\begin{enumerate*}
		\item $a_{i} -a_{i-1} =1$, or 
		\item $\al_j$'s are all of rank zero when $a_{i-1}< j \leq a_i$.  
	\end{enumerate*}
   Suppose $e \leq q-1$. By \eqref{S} in Step 1, if 
   \begin{equation*}
   S\big(\mathcal{B}_1, \dots,\mathcal{B}_t;\, (b\_0, w_0^+), (b\_0, w_0^-)\big) \neq 0,
   \end{equation*}
   then all classes $\al_i$ for $i > e$ lie in one bunch, i.e. $\mathcal{B}_t = \al_{e+1} + \dots + \al_{q}$ and $\mathcal{B}_{t-1} = \al_e$. Also the division of $\al_1, ..., \al_{e-1}$ into $\mathcal{B}_1, \dots, \mathcal{B}_{t-2}$ can be parametrised by non-decreasing surjective maps
		\begin{equation*}
		\psi \colon \{1, ..., e-1 \} \rightarrow \{1, ..., t-2 \}. 
		\end{equation*}
		Thus 
		\begin{equation*}
		U\left(\al_1, ..., \al_{k} ;\, (b\_0, w_0^+), (b\_0, w_0^-)\right) = \sum_{ \substack{
			e \, \in\,  [1, q]\,, \ t' \in [1,\, e-1] \\
			\psi \colon \{1, ...,\, e-1\} \rightarrow \{1, .., t'\} } }  (-1)^{t'}\frac{1}{(q -e)!}\prod_{i=1}^{t'}\frac{1}{\abs{\psi^{-1}(i)}!}.
		\end{equation*}    
		If $e=1$, then $t'=0$ and $\psi = 0$ in the above sum. 
		Finally, the claim \eqref{formula for u} follows by \cite[Equation (72)]{toda:curve-counting-theories-visa-stable-objects-i}.	
	\bigskip
	
\textbf{Step 3.} Now we apply the wall-crossing formula \eqref{wcf}. We know for $i, j \neq e$,
	\begin{equation*}\label{chi-2}
	\chi(\al_i, \al_j) = k_jH\beta_i - k_iH\beta_j = 0 = \chi(\al_j, \al_i). 
	\end{equation*}
	and so 
	\begin{align}\label{chi-1}
	\chi(\al_i, \al_e) = \chi(\al_i, \al). 
	\end{align}
	Thus there exists only one connected digraph $\Gamma$ which gives a non-zero term in the wall-crossing formula \eqref{wcf}: it is made of edges $i \rightarrow e$ for $i< e$ and $e \rightarrow j$ for $j > e$. The corresponding coefficient in \eqref{wcf} is
	\begin{align*}
	& \frac{(-1)^{q-1+\sum_{1 \leq i < j \leq q}\chi(\al_i, \al_j) }}{2^{q -1}} U\left(\al_1, \dots, \al_{q};\, (b_0, w_0^+), (b_0, w_0^-)\right) (-1)^{q-e} \prod_{i \neq e}\chi(\al_i, \al_e)  \\
	& =\  \frac{(-1)^{\sum_{i < e}\chi(\al_i, \al_e)+ \sum_{e < j}\chi(\al_j, \al_e) } }{2^{q-1}(e-1)!(q -e)!} \prod_{i \neq e}\chi(\al_i, \al_e)  \\
	& = \ \frac{1}{2^{q-1}(e-1)!(q -e)!} \prod_{i \neq e}(-1)^{\chi(\al_i, \al_e)}\chi(\al_i, \al_e) 
	\end{align*} 
	For any fixed list of classes $(\al_1, \dots, \al_{e-1} , \al_{e+1}, \dots , \al_q)$, we sum up the above coefficients over the position $e$ of the rank $-1$ class. Since 
	\begin{equation*}
	\sum_{1 \leq e \leq q} \frac{1}{2^{q-1}(e-1)!(q-e)!} = \frac{1}{(q-1)!}\,,
	\end{equation*}
	we obtain
	\begin{align*}
	\J_{b\_0, w_0^-}(\al) = \sum_{
		\substack{
			q \geq 1,\ \al_1, \dots , \al_q \,\in\, \ell\\	
			\al_1 \,\in\, M_{\v, n}\\
			\al_i = (0,\, k_iH,\, \beta_i,\, m_i) \, \text{for $i \in [2, q]$} \\
			0 < k_i < k\\
			\al_1+ \al_2+ \dots + \al_q = \al 
		}
	}\frac{1}{(k-1)!}\J_{b\_0, w_0^+}(\al_1) \prod_{i = 2}^q(-1)^{\chi(\al_i, \al_e)}\chi(\al_i, \al_e) \J_{b\_0, w_0^+}(\al_i)\,. 
	\end{align*}
	Finally, by definition of the line $L_{\alpha}$, there is no wall for rank zero classes $\al_i$ between $\ell$ and the large volume limit. Thus for $i \geq 2$, we have $\J_{b\_0, w_0^+}(\al_i) = \J_{b_0, \infty}(\al_i) = \J(\al_i)$ by Lemma \ref{lem-rank0-tilt-j}. Moreover by \eqref{chi-1}, $\chi(\al_i, \al_e) = \chi(\al_i, \al)$, so the claim \eqref{wcf-prop} follows. 
	
\end{proof}
   
   For any point $(b,w) \in U$ with $b> -n+ \frac{k}{3}$, we define 
   \begin{equation}\label{dtb,w}
   \pt_{b,w}^{\v, n}(x, y, z) \coloneqq \sum_{\alpha_1\,=\, (-1,\, k_1H,\, \beta_1,\, m_1) \,\in\, M_{\v, n}} \J_{b,w}(\alpha_1)\ x^{k_1}y^{\beta_1}z^{m_1} \,.
   \end{equation}	
 As a result of Proposition \ref{prop-wall-crossing-1}, we get the following. 
  \begin{Cor}\label{cor-dt-al}
  	Let $\ell$ be a wall for a class $\al = (-1,\, \tilde{k}H,\, \tilde{\beta},\, \tilde{m}) \in M_{\v, n}$ which lies above $L_{\al}$, and let $(b\_0, w_0^{\pm})$ be points just above and below the wall $\ell$. Then $\J_{b\_0, w_0^-}(\al)$ is equal to the coefficient of $x^{\tilde{k}}y^{\tilde{\beta}}z^{\tilde{m}}$ in the series 
  	\begin{equation*}
  	\pt_{b\_0,\, w_0^+}^{\v, n}(x, y, z)
  	\ \cdot \prod_{
  	\substack{\al' \, =\, (0, \, k'H,\, \beta',\, m') \,\in\, \ell \\
  	0 <\, k'\, < k
  }  
} 
\exp\left((-1)^{\chi(\al', \al)}\chi(\al', \al) \ \J(\al') \, x^{k'}y^{\beta'}z^{m'}  \right). 
  	\end{equation*}
  \end{Cor}

For any rank $-1$ class $\al_1 = (-1,\, k_1H,\, \beta_1,\, m_1) \in K(X)$, we define 
\begin{equation*}
\J_{\infty}(\al) \coloneqq \J_{b,\, w \gg 0\,}(\al) \qquad \text{where $b > \mu_H(\al_1) = -k_1$}.
\end{equation*}
We know $(e^{k_1H}\alpha_1)^{\vee}[1] = \left(1,\ 0,\ -\beta_1- \frac{1}{2}k_1^2H^2,\ m_1+k_1\beta_1.H +\frac{1}{3}k_1^3H^3 \right)$ and $\det(\al) = \cO_X(k)$ because $\Pic(X) = \mathbb{Z}.H$. Thus Lemma \ref{lem.large volume limit} implies   
\begin{equation*}
\J_{\infty}(\alpha_1) = \p_{m_1', \beta_1'}
\end{equation*}
where $\beta_1' \coloneqq \beta_1+ \frac{1}{2}k_1^2H^2$ and $m_1' \coloneqq -m_1 -k_1\beta_1.H -\frac{1}{3}k_1^3H^3$. 

\bigskip 

Take $b > -n+ \frac{k}{3}$ and let $w \rightarrow +\infty$, then the limit of the generating series \eqref{dtb,w} is   
   \begin{align}\label{pt-large volume}
\pt^{\v, n}(x, y, z) \coloneqq & \sum_{\alpha_1\,=\, (-1,\, k_1H,\, \beta_1,\, m_1) \,\in\, M_{\v, n}} \J_{\infty}(\alpha_1)\ x^{k_1}y^{\beta_1}z^{m_1} \,\nonumber \\
= & \sum_{\alpha_1\,=\, -e^{-k_1H}(1,\, 0,\, -\beta_1',\, m_1') \,\in\, M_{\v, n}} \p_{m_1', \beta_1'}\ x^{k_1}y^{\beta_1'-\frac{1}{2}k_1^2H^2}z^{-m_1' -k_1\beta_1'.H +\frac{1}{6}k_1^3H^3} \,.
\end{align}

 
   \begin{Def}\label{Def.A.alpha}
   Let $\al \in K(X)$ be either equal to $\v_n$ or $\al \in M_{\v, n}$. For any 
   real number $\mu \in \mathbb{R}$, we define
	\begin{equation*}
	A(\al, \mu) \coloneqq \pt^{\v, n}(x, y, z)  \ \cdot\  
	\prod_{\substack{  
			\al'\, =\, (0,\, k'H,\, \beta',\, m')\, \in\, K(X)\\
			0 \, <\,  k'\, < \, k \\
			\frac{\beta'.H}{k'H^3} \, > \, \mu
	}}
	\exp\left((-1)^{\chi_{\al'}}\,\chi_{\al'}\ \J(\al') \, x^{k'}y^{\beta'}z^{m'}  \right) 
	\end{equation*}
	where $\chi_{\al'}$ is defined via the following procedure: any non-zero term of $A(\al, \mu)$ is of the form 
	\begin{equation}\label{form.1}
	\J_{\infty}(-1, k_1H, \beta_1, m_1)\, x^{k_1}y^{\beta_1}z^{m_1}\cdot \prod_{
\substack{i\ \in\ [2,\, p] \\
	\alpha_i = (0,\, k_iH,\, \beta_i,\, m_i) 
}	
}	\frac{\left((-1)^{\chi\_{\alpha_i}}\chi\_{\alpha_i} \J(\alpha_i) \, x^{k_i}y^{\beta_i}z^{m_i}  \right)^{q_i}}{q_i !}
	\end{equation} 
	where $q_i \neq 0$ for $i \in [2,\, p]$. We may assume 
	\begin{equation}\label{p}
	\frac{\beta_pH}{k_pH^3} \geq \frac{\beta_{p-1}H}{k_{p-1}H^3}  \geq \dots \geq \frac{\beta_2H}{k_2H^3},
	\end{equation}
	then for any $i \in [2,\, p]$ we define
	\begin{equation}\label{chi}
	\chi\_{\alpha_i} \coloneqq \chi\big(\, \alpha_i\ , \ \al-\sum_{2 \leq\, j\, < i}q_j\,\alpha_j\, \big). 
	\end{equation}  
	   \end{Def}
We know the rank $-1$ destabilising factor $\al_1$ in \eqref{wcf-prop} also lies in $M_{\v, n}$, so we can apply Proposition \ref{prop-wall-crossing-1} to this factor as well. Then finiteness of the number of walls when we move to the large volume limit implies the following. 
\begin{Prop}\label{prop-safe}
	Take a class $\al = (-1, \tilde{k}H, \tilde{\beta}, \tilde{m}) \in M_{\v, n}$ and a point $(b,w) \in U$ above $L_{\alpha}$ which does not lie on a wall for class $\al$. Then $\J_{b,w}(\al)$ is the coefficient of $x^{\tilde{k}}y^{\tilde{\beta}}z^{\tilde{m}}$ in the series $A(\al, \mu(\ell))$ where $\mu(\ell)$ is the gradient of the line $\ell$ passing through $\Pi(\al)$ and $(b,w)$. 
\end{Prop}
\begin{proof}
	We know any non-trivial sentence in $A(\al, \mu(\ell))$ which is $\mathbb{Q}$-multiple of $x^{\tilde{k}}y^{\tilde{\beta}}z^{\tilde{m}}$ is of the form \eqref{form.1} satisfying \eqref{p} and
	\begin{equation*}
	\alpha = (-1, \tilde{k}H, \tilde{\beta}, \tilde{m}) = (-1, k_1H, \beta_1, m_1) + \sum_{i=2}^p q_i(0,\, k_iH,\, \beta_i,\, m_i ). 
	\end{equation*} 
	Therefore there is a wall $\ell_1$ for class $\alpha$ which is made by classes 
	\begin{equation*}
	v_2 \coloneqq \alpha -q_2(0, k_2H, \beta_2, m_2) \qquad \text{and} \qquad w_2 \coloneqq q_2(0, k_2H, \beta_2, m_2).
	\end{equation*}
	The wall $\ell_1 \cap U$ lies above $\ell \cap U$ as it is of slope $\frac{\beta_2H}{k_2H^3} > \mu(\ell)$ and both lines $\ell, \ell_1$ pass through $\Pi(\alpha)$. Moreover \cite[Lemma 3.2]{feyz:rank-r-dt-theory-from-0} implies that 
	\begin{equation*}
	\Delta_H(v_2) < \Delta_H(\alpha).
	\end{equation*}
	 Similarly, there is a numerical wall $\ell_i$ for class 
	 $$
	 v_i \coloneqq \alpha - \sum_{j=2}^{i}q_j(0, k_jH, \beta_j, m_j)
	 $$ 
	 for any $i \in [2, p-1]$ which is made by the destabilising factors  	 
	 \begin{equation*}
	 v_{i+1} = \alpha -\sum_{j=2}^{i+1} q_j(0, k_jH, \beta_j, m_j) \qquad \text{and} \qquad w_{i+1} \coloneqq q_{i+1}(0, k_{i+1}H, \beta_{i+1}, m_{i+1}). 
	 \end{equation*} 
	 Note that the ordering \eqref{p} implies that for any $i \in [1, p-1]$ the wall $\ell_{i} \cap U$ lies above or on $\ell_{i-1} \cap U$ where $\ell_0 = \ell$. Also \cite[Lemma 3.2]{feyz:rank-r-dt-theory-from-0} gives 
	 \begin{equation*}
	 0 \leq \Delta_H(-1, k_1H, \beta_1, m_1) = \Delta_H(v_p) < \Delta_H(v_{p-1}) < \dots < \Delta_H(v_2) < \Delta_H(\alpha).  
	 \end{equation*}
	 
	
	
	\vspace{.2 cm}
	
	To prove the main statement, we proceed by induction on $\Delta_H(\al)$. If $\Delta_H(\al) = 0$, then there is no wall for $\al$ by \cite[Lemma 3.2]{feyz:rank-r-dt-theory-from-0}, thus $\J_{b,w}(\alpha) =\J_{\infty}(\al)$. The above argument also shows that the coefficient of $x^{\tilde{k}}y^{\tilde{\beta}}z^{\tilde{m}}$ in the series $A(\al, \mu(\ell))$ is also $\J_{\infty}(\alpha)$, so the claim follows. 
	   
	
	Now assume the claim holds for all classes $\al \in M_{\v , n}$ with discriminant less than $\Delta_H(\al)$. Since there are only finitely many walls for class $\alpha$ above $L_{\al}$, we can prove the claim for $\alpha$ by induction on walls. 
	
	If $(b,w)$ lies in the large volume limit for $\al$, the claim is trivial. So assume the claim holds for $(b,w^{+})$ just above a wall $\ell_{\alpha}$ for class $\alpha$. 
	Finiteness of the number of walls for each class implies that we may choose $(b, w^{\pm})$ just above and below the wall $\ell_{\al}$ so that they are not on walls for the destabilising classes of $\alpha$ along the wall $\ell_{\alpha}$. Corollary \ref{cor-dt-al} implies that $\J_{b, w^-}(\alpha)$ is the coefficient of $x^{\tilde{k}}y^{\tilde{\beta}}z^{\tilde{m}}$ in the series 
	\begin{equation*}
M \coloneqq \pt_{b,\, w^+}^{\v, n}(x, y, z)\ \cdot \prod_{
	\substack{\al' \, =\, (0, \, k'H,\, \beta',\, m') \,\in\, \ell_{\al} \\
		0 <\, k'\, < k
	}  
} 
\exp\left((-1)^{\chi(\al', \al)}\chi(\al', \al) \ \J(\al') \, x^{c'}y^{\beta'}z^{m'}  \right). 
\end{equation*}
We claim it is equal to the coefficient of $x^{\tilde{k}}y^{\tilde{\beta}}z^{\tilde{m}}$ in the series $A(\alpha, \mu(\ell_{\al}^-))$ where $\ell_{\al}^-$ passes through $\Pi(\al)$ and $(b,w^-)$.  

Any non-zero $Q$-valued multiple of $x^{\tilde{k}}y^{\tilde{\beta}}z^{\tilde{m}}$ in $M$ is of the form 
\begin{equation*}
\J_{b, w^+}(-1, k_1H, \beta_1, m_1)x^{k_1}y^{\beta_1}z^{m_1} \cdot \prod_{
	\substack{i\ \in\ [2,\, p] \\
		\alpha_i = (0,\, k_iH,\, \beta_i,\, m_i) \\
		\frac{\beta_iH}{k_iH^3}\, =\, \mu(\ell_{\al})
	}	
}	\frac{\left((-1)^{\chi(\al_i, \alpha)}\chi(\alpha_i, \alpha) \J_{\infty}(\alpha_i) \, x^{k_i}y^{\beta_i}z^{m_i}  \right)^{q_i}}{q_i !}
\end{equation*}
If $p \geq 2$, we know $\Delta_H(-1, k_1H, \beta_1, m_1) < \Delta_H(\alpha)$. Thus applying the induction on discriminant and on walls implies that $\J_{b, w^-}(\alpha)$ is the coefficient of $x^{\tilde{k}}y^{\tilde{\beta}}z^{\tilde{m}}$ in the series 
	\begin{equation*}
	\left( \sum_{
\substack{
\alpha_1 \in M_{\v,\, n}\\
\Pi(\alpha_1)\, \in\, \ell_{\al}
}	
} A\left(\al_1, \mu(\ell_{\alpha_1}^+)\right) \right)
 \cdot \prod_{
 	\substack{\al_2 \, =\, (0, \, k_2H,\, \beta_2, \, m_2)\\ 
 		0 <\, k_2 \, < k \\
 		\frac{\beta_2H}{k_2H^3}\, =\, \mu(\ell_{\al})
 	}  
 } 
 \exp\left((-1)^{\chi(\al_2, \al)}\chi(\al_2,\al) \ \J(\al_2) \, x^{c_2}y^{\beta_2}z^{m_2}  \right)
	\end{equation*}
	where $\ell_{\alpha_1}^+$ is the line passing through $\Pi(\alpha_1)$ and $(b, w^+)$. Then the coefficient of $x^{\tilde{k}}y^{\tilde{\beta}}z^{\tilde{m}}$ in the above series is the same as the one in $A(\alpha, \mu(\ell_{\al}^-))$
	, so the claim follows.   
\end{proof}
Finally Theorem \ref{thm-rank0-n} follows from the following. 
\begin{Thm}\label{thm-rak--n-section 5}
	The coefficient of $x^{n+k}y^{\beta-\frac{n^2H^2}{2}}z^{m+ \frac{n^3H^3}{6}}$ in the series
	\begin{equation*}
	\frac{(-1)^{\chi(\cO_X(-n), \v) +1}}{\chi(\cO_X(-n), \v)}A(\v_{n_0}, \mu(\ell_f^-))
	\end{equation*}
	is equal to $\J(\v)$. Here $\ell_f^-$ is a line passing through $\Pi(v_{n_0})$ which lies just below $\ell_f$. 
\end{Thm}
\begin{proof}
	First consider the Joyce-Song wall $\ell_{\js}$ for class $\v_n$ that $\cO_X(-n)[1]$ is making. We know the destabilising factors are either of the form 
	\begin{equation*}
	v_1 = [\cO_X(-n)][1] \qquad \text{and} \qquad v_2=\v	
	\end{equation*}
	or 
	\begin{equation}\label{form 2}
	v_1 \in M_{\v, n} \qquad \text{and} \qquad v_2 = (0, k_2H, \beta_2, m_2) \qquad \text{with $0 < k_2 < k$}. 
	\end{equation}
	Let $(b,w^{}\pm)$ be point just above (below) the Joyce-Song wall. Then applying wall-crossing formula \eqref{wcf} and the same argument as in Proposition \ref{prop-wall-crossing-1} imply 
	\begin{align*}
	\J_{b, w^-}(\v_n) = & (-1)^{\chi(\cO_X(-n), \v)}\chi(\cO_X(-n), \v)\, \J_{b, w^+}([\cO_X(-n)[1]])\J_{b, w^+}(\v)\, \\
	&+\sum_{
		\substack{
			q \geq 1,\ \al_1, \dots , \al_q \,\in\, \ell_{\text{JS}}\\	
			\al_1 \,\in\, M_{\v, n}\\
			\al_i = (0,\, k_iH,\, \beta_i,\, m_i) \, \text{for $i \in [2, q]$} \\
			0 < k_i < k\\
			\al_1+ \al_2+ \dots + \al_q = \v_n 
		}
	}\frac{1}{(q-1)!}\J_{b, w^+}(\al_1) \prod_{i = 2}^q(-1)^{\chi(\al_i, \al)}\chi(\al_i, \al) \J(\al_i)\,. 
	\end{align*}
	There is no wall for class $\v$ up to the large volume limit, so $\J_{b, w^+}(\v) = \J_{b, \infty}(\v)$ which is equal to $\J(\v)$ by Lemma \ref{lem-rank0-tilt-j}. Moreover $\cO_X(-n)$ is a rigid objects, so $\J_{b, w^+}([\cO_X(-n)]) =1$ as $\Pic(X) = \mathbb{Z}.H$. We know the destabilising classes along all walls above the Joyce-Song wall are of the form \eqref{form 2}. Thus applying a similar argument as in Proposition \ref{prop-safe} shows $	\J_{b, w^-}(\v_n)$ is the coefficient of $x^{n+k}y^{\beta-\frac{n^2H^2}{2}}z^{m+ \frac{n^3H^3}{6}}$ in the series 
	\begin{equation*}
	(-1)^{\chi(\cO_X(-n), \v)}\chi(\cO_X(-n), \v)\, \J(\v) x^{n+k}y^{\beta-\frac{n^2H^2}{2}}z^{m+ \frac{n^3H^3}{6}} \ +\  A(\v_n, \mu(\ell_{\js}^-)). 
	\end{equation*}  
    For walls below $\ell_{\text{JS}}$, the destabilising factors are all again of the form \eqref{form 2}, so $\J_{b, w_f^-}(\v_n)$ for $(b, w_f^-)$ just below $\ell_f$ is equal to the coefficient of $x^{n+k}y^{\beta-\frac{n^2H^2}{2}}z^{m+ \frac{n^3H^3}{6}}$ in the series 
		\begin{equation*}
	(-1)^{\chi(\cO_X(-n), \v)}\chi(\cO_X(-n), \v)\, \J(\v) x^{n+k}y^{\beta-\frac{n^2H^2}{2}}z^{m+ \frac{n^3H^3}{6}} \ +\  A(\v_n, \mu(\ell_{f}^-)). 
	\end{equation*} 
	But $\J_{b, w_f^-}(\v_n) = 0$ as there is no $\nu_{b, w}$-semistable object of class $\v_n$ for $(b,w)$ below $\ell_f$ and so the claim follows.  
\end{proof}

\section{Rank 2 DT theory}\label{section.rank 2}
The goal of this section is to prove the following. 
\begin{Thm}\label{thm-rank2}
		Let $X$ be a smooth Calabi-Yau 3-fold with $\Pic(X) = \mathbb{Z}.H$. For any rank 2 class $\al \in K(X)$, there is an explicit formula expressing $\J(\al)$ in terms of rank zero and one DT invariants. 
\end{Thm}

We consider two different cases depending on $\ch_1(\al)$ whether it is even or odd multiple of $H$.    


\textbf{Case (i)} Fix a rank 2 class $\w = (2, H, \beta, m )\in K(X)$ with
\beq{vdef}
\ch\_H(\w)\=(2,\,1,\,s,\,d).
\eeq
Then we pick $n\gg0$ and set $\w_n:=\w-[\cO_X(-n)]$, so
\begin{equation*}
\ch\_H(\w_n)\=\Big(\!1, \ n+1 ,\ s-\tfrac12n^2 , \ d +\tfrac16n^3\Big).
\end{equation*}
Consider the line $\ell_f$ given by $B_{b,w}(\w_n) = 0$ \eqref{quadratic form}. One can easily show that for $n \gg 0$, $\ell_f$ intersects $\partial U$ at two points with $b_2^f<0<b_1^f$ such that $b_1^f-b_2^f > n$. 

Let $\ell$ be a wall for an object $E$ of class $\w_n$ with the destabilising sequence $E' \rightarrow E \rightarrow E''$. Then by \cite[Theorem 3.3]{feyz:rank-r-dt-theory-from-0}, one of the factors $E_1$ is a rank one sheaf and the other factor $E_0$ is a rank zero sheaf. We know $\ell$ lies above or on $\ell_f$ and passes through $\Pi(E_1)$, see Figure \ref{figUwm-rank 2}. Thus $0< b_1^f \leq \mu(E_1)$ and so $\ch_1(E_0).H^2 \leq nH^3$. Then the same argument as in \cite[Lemma B.3]{feyz:rank-r-dt-theory-from-0} implies that there is no wall for $E_0$ above $\ell_f$, so we only need to analyse the destabilising factor $E_1$.     

\begin{figure}[h]
	\begin{centering}
		\definecolor{zzttqq}{rgb}{0.27,0.27,0.27}
		\definecolor{qqqqff}{rgb}{0.33,0.33,0.33}
		\definecolor{uququq}{rgb}{0.25,0.25,0.25}
		\definecolor{xdxdff}{rgb}{0.66,0.66,0.66}
		
		\begin{tikzpicture}[line cap=round,line join=round,>=triangle 45,x=1cm,y=0.9cm]
		
		\draw[->,color=black] (-4,0) -- (4,0);
		\draw  (4, 0) node [right ] {$b,\,\frac{\ch_1\!.\;H^2}{\ch_0H^3}$};
		
		\fill [fill=gray!25!white] (0,0) parabola (3,4) parabola [bend at end] (-3,4) parabola [bend at end] (0,0);
		
		
		\draw  (0,0) parabola (3.1,4.27); 
		\draw  (0,0) parabola (-3.1,4.27); 
		\draw  (3.8 , 3.6) node [above] {$w= \frac{b^2}{2}$};
		
		\draw[->,color=black] (0,-2) -- (0,4.7);
		\draw  (1, 4.3) node [above ] {$w,\,\frac{\ch_2\!.\;H}{\ch_0H^3}$};
		
		\draw [dashed] (-2.45,2.65) -- (-2.45,0);
		\draw [dashed] (1.8, -1.5) -- (1.8,4);
		\draw [thick] (1.8, -1.5) -- (-3.8,4);
		\draw  (1.8, -1.5) -- (-1,4);
		
		
		\draw  (-2.45, 0) node[below]{$b_2^f$};
		\draw  (.25, 0) node[below]{$b_1^f$};
		
		\draw  (1.8,-1.5) node[right] {{\small$\Pi(\w_n)$}};
		\draw (-3.8,4) node [left]{$\ell_{f}$};
		\draw (-.95, 4) node [left]{$\ell$};
		\draw  (2.4, 3.2) node [above] {\Large{$U$}};
		\fill (1.8,-1.5) circle (2pt);
		\fill (-2.45,2.65) circle (1.5pt);
		\fill (.25,.03) circle (1.5pt);
		
		\end{tikzpicture}
		
		\caption{The final wall $\ell_{f}$ for class $\w_n$}
		\label{figUwm-rank 2}
	\end{centering}
\end{figure}
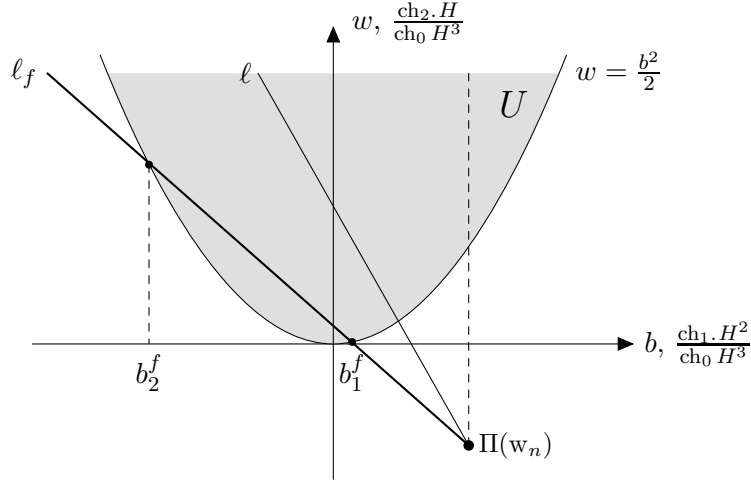


Define $\widetilde{M}_{\w, n} \subset K(X)$ to be the set of all rank one classes $\al = (1, \ch_1, \ch_2, \ch_3) \in K(X)$ with $\Delta_H(\al) \geq 0$ such that 
\begin{equation*}
0 < \frac{\ch_1H^2}{H^3} < n+1 \qquad \text{and} \qquad \text{$\Pi(\al)$ lies above or on $\ell_{f}$}. 
\end{equation*}  
As explained in \cite[Section 3]{feyz:rank-r-dt-theory-from-0}, for any class $\alpha \in \widetilde{M}_{\w, n}$, there is a unique line $\ell_{\alpha}$ going through $\Pi(\alpha)$ which intersects $\partial U$ at two points with $b$-values $a_v < b_v \leq \mu_H(\al)$ satisfying
\begin{equation*}\label{l-v}
H^3(b_v -a_v) \= \ch_1(\alpha)H^2 -b_vH^3\ \geq\ 0.
\end{equation*}
The area above the line $\ell_{\al}$ is called the \emph{safe area} for $\alpha$ and denoted by
\begin{equation}\label{u-al}
U_{\alpha} \coloneqq \left\{(b,w) \in U\colon b < \frac{\ch_1(\alpha)H^2}{H^3} \ \ \text{and} \ \text{$(b,w)$ lies above $\ell_{\al}$} \right\}. 
\end{equation}
Take an object $E \in \cA_b$ of class $\al \in \widetilde{M}_{\w, n}$ which is $\nu_{b,w}$-semistable for some $(b,w)\in U_{\al}$. Then \cite[Proposition 3.1]{feyz:rank-r-dt-theory-from-0} implies that 
\begin{enumerate*}
	\item $E$ is a sheaf, 
	\item if $E_1 \hookrightarrow E \twoheadrightarrow E_2$ is a short exact sequence in $\cA_{\;b}$ with $\nu\_{b,w}(E_1)=\nu\_{b,w}(E_2)$ then one of $E_i$'s is a rank zero tilt-semistable sheaf, and the other factor $E_j$ is a rank one sheaf such that $\ch(E_j) \in \widetilde{M}_{\w, n}$ and $(b,w) \in U_{[E_j]}$. 
 \end{enumerate*}
Let $\ell$ be a wall for class $\w_n$, then \cite[Theorem 3.3]{feyz:rank-r-dt-theory-from-0} implies that the destabilising factors are either of the form 
\begin{enumerate*}
	\item $v_1 = [\cO_X(-n)][1], v_2 = \w$ and there is no wall for class $\w$ between $\ell$ and the large volume limit, 
	\item $v_1 \in \widetilde{M}_{\w, n}, v_2 = (0, k_2H, \beta_2, m_2)$ with $\ell \in U_{v_1}$ and there is no wall for class $v_2$ between $\ell_f$ and the large volume limit.  
\end{enumerate*} 
Consider the generating series  
\begin{align*}
\dt^{\w, n}(x, y, z) \coloneqq & \sum_{\alpha_1\,=\, (1,\, k_1H,\, \beta_1,\, m_1) \,\in\, \widetilde{M}_{\w, n}} \J(\alpha_1)\ x^{k_1}y^{\beta_1}z^{m_1} \\
= & \sum_{\alpha_1\,=\, e^{k_1H}(1,\, 0,\, -\beta_1',\, -m_1') \,\in\, \widetilde{M}_{\w, n}} \qe_{m_1', \beta'_1}\ x^{k_1}\,y^{-\beta'_1 + \frac{1}{2}k_1^2H^2}\,z^{-m_1'-k\beta_1'H+\frac{1}{6}k^3H^3}\,.
\end{align*}
\begin{Def}\label{Def.A.tilde-alpha}
	Let $\al \in K(X)$ be either equal to $\w_n$ or $\al \in \widetilde{M}_{\w, n}$. For any real number $\mu \in \mathbb{R}$, we define
\begin{equation*}
\widetilde{A}(\al, \mu) \coloneqq \dt^{\w, n}(x, y, z) \cdot \  
\prod_{\substack{  
		\al'\, =\, (0,\, k'H,\, \beta',\, m')\, \in\, K(X)\\
		0 \, <\,  k'\, \leq n \\
		\frac{\beta'.H}{k'H^3} \, < \, \mu
}}
\exp\left((-1)^{\chi_{\al'}}\,\chi_{\al'}\ \J_{\infty}(\al') \, x^{c'}y^{\beta'}z^{m'}  \right) 
\end{equation*}
	where $\chi_{\al'}$ is defined via the following procedure: any non-zero term of $\widetilde{A}(\al, \mu)$ is of the form 
	\begin{equation}\label{form}
	\J(1, k_1H, \beta_1, m_1)\, x^{k_1}y^{\beta_1}z^{m_1}\cdot \prod_{
		\substack{i\ \in\ [2,\, p] \\
			\alpha_i = (0,\, k_iH,\, \beta_i,\, m_i) 
		}	
	}	\frac{\left((-1)^{\chi\_{\alpha_i}}\chi\_{\alpha_i} \J(\alpha_i) \, x^{k_i}y^{\beta_i}z^{m_i}  \right)^{q_i}}{q_i !}
	\end{equation} 
	where $q_i \neq 0$ for $i \in [2,\, p]$. We may assume 
	\begin{equation}\label{p-rank2}
	\frac{\beta_pH}{k_pH^3} \leq \frac{\beta_{p-1}H}{k_{p-1}H^3}  \leq \dots \leq \frac{\beta_2H}{k_2H^3},
	\end{equation}
	then for any $i \in [2,\, p]$ we define
	\begin{equation}\label{chi-rank2}
	\chi\_{\alpha_i} \coloneqq \chi\big(\, \alpha_i\ , \ \al-\sum_{2 \leq\, j\, < i}q_j\,\alpha_j\, \big). 
	\end{equation}  
\end{Def}
Applying the same argument as in Proposition \ref{prop-safe} implies the following.
\begin{Lem}
	Take a class $\al = (1, \tilde{k}H, \tilde{\beta}, \tilde{m}) \in \widetilde{M}_{\w, n}$ and a point $(b,w) \in U_{\alpha}$ which does not lie on a wall for class $\al$. Then $\J_{b,w}(\al)$ is the coefficient of $x^{\tilde{k}}y^{\tilde{\beta}}z^{\tilde{m}}$ in the series $\widetilde{A}(\al, \mu(\ell))$ where $\mu(\ell)$ is the gradient of the line $\ell$ passing through $\Pi(\al)$ and $(b,w)$. 
\end{Lem}
As a result, we can prove Theorem \ref{thm-rank2} in odd cases: 
\begin{Prop}\label{porp-rank2-final}
	 Consider a rank 2-class $\al \in K(X)$ with $\ch_1(\al) =kH$ where $\gcd(2, k) =1$. Let $\w = e^{-\frac{k-1}{2}H}\al = (2, H, \beta, m)$ and $\w_n = \w-[\cO_X(-n)]$ for $n \gg 0$. Then the coefficient of $x^{n+1}y^{\beta-\frac{n^2H^2}{2}}z^{m+ \frac{n^3H^3}{6}}$ in the series
	\begin{equation*}
	\frac{(-1)^{\chi(\cO_X(-n), \w) +1}}{\chi(\cO_X(-n), \w)} \widetilde{A}(\w_n, \mu(\ell_{f}^-))
	\end{equation*}
	is equal to $\J(\al)$. Here $\ell_{f}^-$ is a line passing through $\Pi(\w_n)$ which lies just below $\ell_{f}$. 
\end{Prop}
\begin{proof}
	We know 
	any tilt-semistable sheaf of class $\al$ or $\w$ is slope-stable, so it is in particular H-Gieseker stable. Therefore $\J(\al) = \J(\w) = \J_{\mathrm{ti}}(\w)$. Since all walls for class $\w_n$ lie above or on $\ell_{f}$, the same argument as in Theorem \ref{thm-rak--n-section 5} implies the claim.
\end{proof}

\textbf{Case (ii)} Now fix a rank $2$-class 
\begin{equation*}
\w =  (2, 0, \beta, m).
\end{equation*}
Then pick $n \gg 0$ and let $\w_n = \w-[\cO_X(-n)]$ as before. \cite[Theorem B.2]{feyz:rank-r-dt-theory-from-0} implies that there is no $\nu_{b,w}$-semistable object of class $\w_n$ below the Joyce-Song wall $\ell_{\js}$ which $\cO_X(-n)[1]$ is making. 

We first examine walls for class $\w_n$ above $\ell_{\js}$. Define $\widetilde{M}_{\w, n} \subset K(X)$ to be the set of all rank one classes $\al = (1, \ch_1, \ch_2, \ch_3) \in K(X)$ with $\Delta_H(\al) \geq 0$ such that 
\begin{equation*}
0 \leq \frac{\ch_1H^2}{H^3} < n \qquad \text{and} \qquad \text{$\Pi(\al)$ lies above or on $\ell_{\js}$}. 
\end{equation*}
For any class $\alpha \in \widetilde{M}_{\w, n}$, we consider the open subset $U_{\al}$ as in \eqref{u-al}. Let $\ell$ be a wall for class $\w_n$ above $\ell_{\js}$, then \cite[Theorem 3.3]{feyz:rank-r-dt-theory-from-0} implies that the destabilising factors are of the form
\begin{equation*}
v_1 \in \widetilde{M}_{\w, n} \qquad \text{and} \qquad v_2 = (0, k_2H, \beta_2, m_2)
\end{equation*}
such that $\ell \in U_{v_1}$ and there is no wall for class $v_2$ between $\ell_{\js}$ and the large volume limit\footnote{Note that in case (2)(b) of \cite[Theorem 3.3]{feyz:rank-r-dt-theory-from-0}, we apply \cite[Theorem B.1]{feyz:rank-r-dt-theory-from-0} which says there is no wall for our rank zero class $v_2$ above the wall that $\cO_X(-n)$ is making. }. Thus we get the following. 
\begin{Lem}
	Let $(b,w^+) \in U$ be a point just above $\ell_{\js}$. Then $\J_{b, w^+}(\w_n)$ is the coefficient of $x^{n}y^{\beta-\frac{n^2H^2}{2}}z^{m+ \frac{n^3H^3}{6}}$ in the series $\widetilde{A}(\al, \mu(\ell_{\js}^+))$ where $\ell_{\js}^+$ is a line passing through $\Pi(\w_n)$ which lies just above $\ell_{\js}$ 
\end{Lem}
Along the Joyce-Song wall, there are three possibilities for the destabilising factors:
\begin{enumerate*}
	\item $v_1 = \w$ and $v_2 = [\cO_X(-n)][1]$,
	\item $v_1 = (1, 0, \beta', m')$ and $v_2 = \left(0,\, nH,\, \beta-\beta' -\frac{n^2H^2}{2},\, m-m'+\frac{n^3H^3}{6}\right)$ such that $\beta'.H = \frac{1}{2}\beta.H$,
	\item $v_1 = (1, 0, \beta', m')$, $v_1' = (1, 0, \beta'', m'')$ and $v_2 = [\cO_X(-n)][1]$ such that $v_1+v_1'+v_2 = \w_n$ and $\beta'.H = \beta''.H = \frac{1}{2}\beta.H$.   
\end{enumerate*} 
Let $(b, w^{\pm})$ be points just above (below) $\ell_{\js}$. We know there is no wall for all the above destabilising factors between $(b,w^+)$ and the large volume limit. Applying the wall-crossing formula \eqref{wcf} gives 
\begin{align}\label{w}
\J_{b, w^-}(\w_n) = & \J_{b, w^+}(\w_n) + (-1)^{\chi(\cO_X(-n), \w)}\chi(\cO_X(-n), \w)\, \J_{\mathrm{ti}}(\w) + A\_{\w_n} 
\end{align}
where 
\begin{align*}
A_{\w_n} \coloneqq &\ \sum_{\substack{
v_1 = (1, 0, \beta', m'), \ 2\beta'.H = \beta.H\\
v_2 = \left(0,\, nH,\, \beta-\beta' -\frac{n^2H^2}{2},\, m-m'+\frac{n^3H^3}{6}\right) 
}} (-1)^{\chi(v_2, v_1)+1}\chi(v_2, v_1)\J(v_1)\J(v_2) \nonumber \\
&\ + \sum_{\substack{
v_1 = (1, 0, \beta', m')\\
v_1' = (1, 0, \beta'', m'')\\
v_1+v_1'-[\cO_X(-n)] = \w_n \\
2\beta'.H = 2\beta''.H = \beta.H
}} C_{v_1, v_1'}\J(v_1)\J(v_1')\, .\nonumber 
\end{align*}
Here $C_{v, v'} \in \Q$ depends on $v, v'$ and can be explicitly determined by \eqref{wcf}.  
\begin{Prop}\label{prop-rank2-final-case ii}
	Consider a rank 2-class $\al \in K(X)$ with $\ch_1(\al) =kH$ where $\gcd(2, k) =2$. Let $\w = e^{-\frac{k}{2}H}\al = (2, 0, \beta, m)$ and $\w_n = \w-[\cO_X(-n)]$ for $n \gg 0$. Then $\J_{\mathrm{ti}}(\al)$ is equal to  
	\begin{equation*}
	\frac{(-1)^{\chi(\cO_X(-n), \w) +1}}{\chi(\cO_X(-n), \w)} \left(\J_{b, w^+}(\w_n) + A\_{\w_n}  \right)
	\end{equation*}
	where $\J_{b, w^+}(\w_n)$ is the coefficient of $x^{n}y^{\beta-\frac{n^2H^2}{2}}z^{m+ \frac{n^3H^3}{6}}$ in the series
	\begin{equation*}
	\frac{(-1)^{\chi(\cO_X(-n), \w) +1}}{\chi(\cO_X(-n), \w)} \widetilde{A}(\w_n, \mu(\ell_{\js}^+)).
	\end{equation*}
\end{Prop}
The final step is to apply the wall-crossing formula \eqref{wcf-tilt-gieseker} between tilt-stability and Gieseker-stability. Since $\al$ is of rank 2, there are at most two factors $\alpha_1, \alpha_2$ and we know they have the same tilt-slope, so \eqref{sentence-2-terms} implies
\begin{equation*}
\J_{\text{ti}}(\al) = \J(\al) + \sum_{\substack{
\alpha_1 = e^{\frac{k}{2}H}(1, 0, -\beta_1, -m_1) \in C(\Coh(X)) \\
\alpha_2 = e^{\frac{k}{2}H}(1, 0, -\beta_2, -m_2) \in C(\Coh(X))\\
\al_1+\al_2 =\al\\
\beta_1.H =\beta_2.H, \ 
m_1 > m_2}} 
 (-1)^{m_1-m_2}(m_1-m_2)\qe_{m_1, \beta_1}\qe_{m_2, \beta_2}.
\end{equation*}
 Combining this with Proposition \ref{prop-rank2-final-case ii} completes the proof of Theorem \ref{thm-rank2}. 
 
\begin{Rem}\label{rem.final}
	The results in Section \ref{js section} and Section \ref{section.rank 2} are based on the wall-crossing arguments in \cite{feyz:rank-r-dt-theory-from-0,feyz:rank-r-dt-theory-from-1}. As discussed in these papers, the restricted set of weak stability conditions handled in \cite{chunyi:stability-condition-quintic-threefold,koseki:double-triple-solids,liu:bg-ineqaulity-quadratic} are sufficient for our purposes, thus Theorem \ref{thm-rank0-n} and Theorem \ref{thm-rank2} are valid in these cases. 
\end{Rem}

\bibliography{mybib}
\bibliographystyle{halpha}

\bigskip \noindent
{\tt{s.feyzbakhsh@imperial.ac.uk}}\medskip

\noindent Department of Mathematics, Imperial College, London SW7 2AZ, United Kingdom

\end{document}